\documentclass[a4paper,12pt]{report}

\usepackage{amsmath}
\usepackage{amsthm}
\usepackage{amsfonts}
\usepackage{amssymb}

\parskip\smallskipamount
\numberwithin{equation}{section} 

\newtheorem{thm}{Theorem}[chapter]
\newtheorem{lem}[thm]{Lemma}
\newtheorem{cor}[thm]{Corollary}
\newtheorem{prop}[thm]{Proposition}
\newtheorem{definition}{Definition}
\newtheorem{remark}{Remark}

\newcommand{\pr}{\mathbf{P}}
\newcommand{\E}{\mathbf{E}}
\newcommand{\R}{\mathbb{R}}

\newcommand{\Z}{\mathbb{Z}}
\newcommand{\T}{\mathbb{T}}
\newcommand{\G}{\mathbb{G}}
\newcommand{\I}{\mathcal{I}}
\renewcommand{\L}{\mathbb{L}}

\newcommand{\M}{\mathcal{M}}
\newcommand{\Zm}{\Z^{\mbox{\bf{-}}}}
\newcommand{\Zmg}{Z^{\mbox{\bf{-}}}}

\title{\sc{CONTACT PROCESSES ON THE INTEGERS}}  
\author{  \textup{Achillefs Tzioufas} \vspace{10mm} \\ 
  \sc{Awarded the degree of Doctor of Philosophy} \\
      \sc{ on completion of research in the Department of }\\
        \sc{ Actuarial Mathematics and Statistics, School of }\\
         \sc{ Mathematical and Computer Sciences,} \\
        \sc{ Heriot-Watt University}\\
    \vspace{25mm}\\ 
    \sc{Advisor:} \textup{Stan Zachary } \vspace{5mm} \\  
\sc{ External Examiner:} \textup{Thomas Mountford }\\
\sc{ Internal Examiner }: \textup{Sergey Foss } 
\vspace{15mm} \\
}

\date{\sc{November} 2011}

\begin{document}

\pagenumbering{roman}

\maketitle
\chapter*{Abstract}
\thispagestyle{empty}
The three state contact process is the modification of the contact process at rate $\mu$ in which first infections occur at rate $\lambda$ instead. Chapters 2 and 3 consider the three state contact process on (graphs that have as set of sites) the integers with nearest neighbours interaction (that is, edges are placed among sites at Euclidean distance one apart). Results in Chapter 2 are meant to illustrate regularity of the growth of the process under the assumption that $\mu \geq \lambda$, that is, reverse immunization. While in Chapter 3 two results regarding the convergence rates of the process are given.  Chapter 4 is concerned with the i.i.d.\ behaviour of the right endpoint of contact processes on the integers with symmetric, translation invariant interaction. Finally, Chapter 5 is concerned with two monotonicity properties of the three state contact process. 
\chapter*{Acknowledgments}
\thispagestyle{empty}
The author wishes to thank Stan Zachary for research advising and valuable discussions on parts of this work. He also wishes to thank Thomas Mountford for comments on the previous version, and especially for pointing out an error in Theorem 3.4 thereof. This work was financially supported in part by a Heriot-Watt University studentship.



\newpage
\pagenumbering{roman}\tableofcontents

\chapter{Introduction} 

\pagenumbering{arabic}
The purpose of this chapter is to give an overview of all important results in the thesis and their relation to other work in the field, which we firstly report. An introductory section is also given in each of the remaining chapters.


\section{Model description}

The three state contact process is described as a model for the spread of a contagious disease over a population, which is comprised by individuals located at the set sites of a simple and locally finite graph. If an edge is present at the graph the individuals associated to the sites are considered as neighbouring. The disease is transmitted through contacts between infected individuals and their susceptible (healthy) neighbours, hence the name of the process. One can also think of contacts as the emission of a pathogenic bacterium or a microscopic particle. 

At all times each individual is assigned one of the three plausible states: infected, susceptible and never infected, and, susceptible and previously infected. That is, individuals can be either infected or susceptible, however, there are two states of susceptibility. At any time the overall state of the process is given by the collection of all states, which is referred to as the configuration of the process.

The infinitesimal dynamics for the evolution of the \textit{three state contact process} can then be described as follows: (i) Infected individuals make contact at rate $\lambda$ or $\mu$ with each of their neighbouring susceptible individuals depending on whether they are susceptible and never infected or susceptible and previously infected, respectively. (ii) At rate 1 infected individuals recover, hence becoming susceptible and previously infected. (For the purposes of modeling, note that (i) imposes that infected individuals are simultaneously infectious). 




The case that $\lambda = \mu$, where the two susceptibility states are effectively equated, is the well-known \textit{contact process}. Most of the motivation for studying the three state contact process stems mostly from the interest in understanding the change induced to the behaviour of the contact process by allowing first infections to occur at a different rate. Furthermore, in the boundary case $\mu=0$ the process reduces to the well-known \textit{standard spatial epidemic}; in this case individuals are are effectively removed from the population when becoming susceptible and previously infected for the first time (i.e.\ when recovery from their first infection occurs), hence the standard spatial epidemic is also referred to as the susceptible-infected-removed model in the literature.

Regarding the underlying structure of the population, in all chapters of the thesis but the last one the focus is on the process on simple graphs with set of sites the integers for which an edge is placed between sites at Euclidean distance not greater to some fixed finite constant. If this constant is specified to be one, the corresponding process is referred to as the process \textit{on the integers with nearest neighbours interaction}. If not, the corresponding family of processes (indexed by the constant) is referred to as the processes \textit{on the integers with symmetric, translation invariant interaction}. The last case is considered only in Chapter 4. To avoid confusion we should mention that the contact process on the integers with nearest neighbours interaction is usually referred to in the literature as the basic one-dimensional contact process; see e.g.\ part 1 in \cite{D95}. Also in this more general setting contact processes on the integers with symmetric, translation invariant interaction as considered here are referred to as uniform, symmetric, translation invariant, finite range one-dimensional contact processes.  


 
\section{Known results}

%

Some concepts that capture characteristics of interest exhibited by the model are introduced.  Typically, an infected individual is placed in the midst of susceptible and never infected individuals; this starting state for the process is referred to as the standard initial configuration. The process is said to \textit{survive} when there is a positive probability that the disease will be present in the population for all times, while otherwise, the process is said to \textit{die out}. 

Let $\mathbb{L}^{d}$ be the graph with set of sites the $d$-dimensional integer lattice at which edges are placed among sites at Euclidean distance one apart.  It is known that the contact process on $\mathbb{L}^{d}$ exhibits a phase transition phenomenon for all $d\geq1$. That is, there is a critical value $\mu_{c}(d)$, $0<\mu_{c}(d)<\infty$, such that if $\mu< \mu_{c}(d)$, the process dies out, while if $\mu> \mu_{c}(d)$, the process survives. For an account of various results and proofs concerning the widely studied contact process we refer the reader to \cite{L}, \cite{D88} and \cite{L99}. Furthermore, it is known that the the standard spatial epidemic on $\mathbb{L}^{d}$ exhibits a phase transition phenomenon if and only if $d\geq2$, that is, letting $\lambda_{c}(d)$ denote the critical value of the standard spatial epidemic on $\mathbb{L}^{d}$, it is known that $0<\lambda_{c}(d)<\infty$ for all $d\geq2$, while $\lambda_{c}(1)=\infty$.  For standard results and proofs about this process we refer the reader to Chapter 9 in \cite{D88}. Note that the process on $\mathbb{L}^{1}$ is referred to here as the process on the integers with  neighbours interaction.


The three state contact process on $\mathbb{L}^{d}$ has been studied in Durrett and Schinazi \cite{DS} and in Stacey \cite{S}. The process is said to survive strongly when there is a positive probability that the initially infected individual is reinfected infinitely many times.  Theorem 4 in \cite{DS} implies that if $\mu > \mu_{c}(d)$ and $\lambda>0$, the process survives strongly. While Theorem 1.1 in \cite{S} provides that if $\mu < \mu_{c}(d)$ and $\lambda<\infty$, the process does not survive strongly. Furthermore, Theorem 3 in \cite{DS} provides that the process on the integers with nearest neighbours interaction (in other words, the process on $\mathbb{L}^{1}$) for $\mu < \mu_{c}(1)$  and $\lambda<\infty$ dies out---note however that for $d\geq2$ the process survives for $\lambda$ sufficiently large and $\mu=0$ since $\lambda_{c}(d)<\infty$; see also Theorem \ref{ForestFire} below for a related result.  


Finally, some results concerning the supercritical contact process on the integers with nearest neighbours interaction conditional on survival are described. Let  $r_{t}$ denote the rightmost infected site of the process at time $t$. Durrett \cite{D80} showed that the limit of the speed of $r_{t}$ is a strictly positive constant, that is $\displaystyle{\frac{r_{t}}{t} \rightarrow \alpha}$ as $t \rightarrow \infty$ and $\alpha>0$, almost surely. Galves and Presutti \cite{GP} proved the corresponding central limit theorem, that is, the limit as $t \rightarrow \infty$ of the distribution of $\displaystyle{\frac{r_{t}-\alpha t}{\sqrt{t}}}$ is the normal distribution. Later, Kuczek \cite{K} showed the existence of random space-time points, referred to as break points, at which regeneration of the stochastic behaviour of $r_{t}$ occurs, thus illustrating the i.i.d.\ behaviour of the growth of $r_{t}$. By additionally showing that break points occur at exponentially bounded distance, an alternative proof of the central limit theorem was provided there. Furthermore, Mountford and Sweet \cite{MS} adapted the argument of Kuczek and extended the central limit theorem for one-dimensional non-nearest neighbours, finite range contact processes (which is a class of processes that includes the contact processes on simple, locally finite graphs with set of sites the integers). 

\section{Overview of the thesis} 
In Chapter 2 the three state contact process on the integers with nearest neighbours interaction for $\mu>\mu_{c}(1)$ and $\lambda \leq \mu$ started from the standard initial configuration is considered. Theorem 4 in Durrett and Schinazi \cite{DS} implies in particular that the process survives; results given in this chapter are meant to illustrate regularity of the growth of the process conditioned to survive. The strong law and the central limit theorem for the rightmost infected site of the process are the main results. The technique of proof is based on the adaptation of Kuczek's argument to show a new type of break points given in Section 2.5. Also, for showing that the points occur at exponentially bounded distance, two exponential estimates for the process are proved in Section 2.4 by use of the comparison with oriented percolation result in \cite{DS}. Chapter 3 also considers the three state contact process on the integers with nearest neighbours interaction. For the case that $\mu>\mu_{c}(1)$ and $\lambda\leq \mu$, it is shown that the limit of the speed of the process is bounded above by $\lambda / \mu$ times that of the contact process on the integers with nearest neighbours interaction at rate $\mu$. For the case that $\mu<\mu_{c}(1)$ and $\lambda<\infty$, it is shown that the time to die out is exponentially bounded, thus, Theorem 3 in \cite{DS} is extended by a different method of proof.

Chapter 4 considers contact processes on the integers with symmetric, translation invariant interaction. In this chapter an elementary proof of the i.i.d.\ behaviour of the rightmost infected of the process that uses Mountford and Sweet's \cite{MS} adaptation of Kuczek's argument is presented. Based on a result from Durrett and Schonmann \cite{DS2}, a large deviations result for the set of infected sites related to the work in \cite{MS} is also given. 
Finally, Chapter 5 is concerned with two monotonicity properties of the three state contact process.



\chapter{On the growth of reverse immunization contact processes}\label{pap1}

\paragraph{Abstract:} The variation of the supercritical contact process on the integers with nearest neighbours interaction such that first infection occurs at a lower rate is considered; Theorem 4 in Durrett and Schinazi \cite{DS} implies in particular that the process survives with positive probability. Regarding the rightmost infected of the process started from one site infected and conditioned to survive, we specify a sequence of space-time points at which its behaviour regenerates and, thus, obtain the corresponding strong law and central limit theorem. We also extend complete convergence in this case. 

\vfill


\section{Introduction and main results}\label{S0}
We begin by defining a class of processes that includes the processes we are especially interested in. The three state contact process on the integers  with nearest neighbours interaction and parameters $(\lambda,\mu)$ is a continuous time Markov process $\zeta_{t}$ with state space $\{-1,0,1\}^{\Z}$, elements of which are called configurations and can be thought of as functions from $\Z$ to $\{-1,0,1\}$. The evolution of $\zeta_{t}$ is described locally as follows, transitions at each site $x$, $\zeta_{t}(x)$, occur according to the rules:
\begin{equation*}\label{rates}
\begin{array}{cl}
-1 \rightarrow 1 & \mbox{ at rate } \lambda \hspace{0.5mm} |\{ y = x-1,x+1 :\zeta_{t}(y)= 1\}|, \\
\mbox{ } 0 \rightarrow 1 & \mbox{ at rate } \mu \hspace{0.5mm} |\{ y = x-1,x+1 :\zeta_{t}(y)= 1\}|, \\   
\mbox{ } 1 \rightarrow 0 & \mbox{ at rate } 1,
\end{array}
\end{equation*}
for all times $t \geq0$, where $|B|$ denotes the cardinal of $B\subset \Z$.   Typically, the process started from configuration $\eta$ is denoted as $\zeta_{t}^{\eta}$. For general information about interacting particle systems such as the fact that the above rates specify a well-defined process, see Liggett \cite{L}. We note that the cases $\lambda = \mu$ and $\mu=0$ correspond to the extensively studied processes known as the contact process and as the forest fire model, respectively; an account of various related results and proofs can be found in Chapters 4, 9, and 10 of Durrett \cite{D88} and Part I of Liggett \cite{L99}. Furthermore in the literature various survival aspects of the three state contact process on the $d$-dimensional lattice were studied by Durrett and Schinazi \cite{DS} and by Stacey \cite{S}, the latter also includes results for the process on homogeneous trees. 


The process is thought of according to the following epidemiological interpretation. Given a configuration $\zeta$, each site  $x$  is regarded as infected if $\zeta(x) = 1$, as susceptible and never infected if $\zeta(x) = -1$ and, as susceptible and previously infected if $\zeta(x) = 0$. The \textit{standard initial configuration} is such that the origin is infected while all other sites are susceptible and never infected. We will use  $\zeta_{t}^{O}$ to denote the process started from the standard initial configuration. 
We say that the three state contact process \textit{survives} if $\pr(\zeta_{t}^{O} \mbox{ survives})>0$, where the event $\{\forall \hspace{0.5mm}t\geq0, \exists  \hspace{0.5mm} x: \zeta_{t}(x) =1\}$ is abbreviated as $\{\zeta_{t} \mbox{ survives}\}$. 

Note that for $\zeta_{t}^{O}$ transitions $-1 \rightarrow 1$,  $ 0 \rightarrow 1$, and $1 \rightarrow 0$ correspond respectively to initial infections,  subsequent infections and recoveries.  Accordingly, the initial infection of a site
induces a permanent alternation of the rate proportional to which it will be susceptible; thus, the rate either decreases, corresponding to (partial) immunization, or increases, i.e.\ the reverse occurs. Our results concern the three state contact process under the constraint that $\mu\geq\lambda$, this explains the title given to this chapter.
When modeling an epidemic,  the case that $\mu \leq \lambda$ could be a consequence of imperfect inoculation of individuals following their first exposure to the disease, while the case that $\mu \geq \lambda$ could be a consequence of debilitation of individuals caused by their first exposure to the disease. Specifically, tuberculosis and bronchitis are plausible examples of a disease that captures the latter characteristic.



When $(\lambda,\mu)$ are such that $\lambda = \mu$ the process is reduced to the well known contact process. In this case we will identify a configuration with the subset of $\Z$ that corresponds to the set of its infected sites, since states $-1$ and $0$ are effectively equivalent. Also, it is well known that the contact process exhibits a phase transition phenomenon, $\mu_{c}$ will denote critical value on the integers with nearest neighbours interaction, that is, $0<\mu_{c}<\infty$ and, if $\mu< \mu_{c}$ the process dies out while if $\mu> \mu_{c}$ the process survives. 


The three state contact process with parameters $(\lambda,\mu)$ such that $\mu > \mu_{c}$ and $\lambda>0$ is known to survive, see Durrett and Schinazi \cite{DS}. This chapter is concerned with the behaviour of the process when survival occurs and also $\mu\geq\lambda$. The reason for the additional assumption on the parameters is that the techniques of proof extensively use certain coupling results which hold in this case only (see section \ref{coupling}). 

The following theorem summarizes the main results of this chapter, 
in words, parts \textit{(i)} and \textit{(ii)}  are respectively a law of large numbers and the corresponding central limit theorem for the rightmost infected while parts \textit{(iii)} and \textit{(iv)} are respectively a law of large numbers and complete convergence for the set of infected sites of the process. For demonstrating our results we introduce some notation.  The standard normal distribution function is represented by $N(0,\sigma^{2}), \sigma^{2}>0,$ also, weak convergence of random variables and of set valued processes are denoted by "$\overset{w}{\rightarrow}$" and by "$\Rightarrow$" respectively. Further, we denote by $\bar{\nu}_{\mu}$ the upper invariant measure of the contact process with parameter $\mu$,  and by $\delta_{\emptyset}$ the probability measure that puts all mass on the empty set. (For general information about the upper invariant measure and weak convergence of set valued processes we refer to  pages 34-35 in Liggett \cite{L99}).
\begin{thm}\label{THEthm1} 
Consider $\zeta_{t}^{O}$ with parameters $(\lambda, \mu)$, and let $I_{t}= \{x: \zeta_{t}^{O}(x)=1\}$ and $r_{t} = \sup I_{t}$. If $(\lambda, \mu)$ are such that $\mu \geq \lambda>0$ and $\mu > \mu_{c}$ then there exists $\alpha>0$ such that conditional on $\{\zeta_{t}^{O} \mbox{\textup{ survives}}\}$,

(i) $\displaystyle{  \frac{r_{t}}{t} \rightarrow \alpha}$, almost surely;

(ii) $\displaystyle{ \frac{r_{t} - \alpha t}{\sqrt{t}} \overset{w}{\rightarrow} N(0,\sigma^{2})}$, for some $\sigma^{2}>0$;

(iii) let $\theta= \theta(\mu)$ be the density of $\bar{\nu}_{\mu}$, then, $\displaystyle{ \frac{|I_{t}|}{t} \rightarrow 2 \alpha \theta}$, almost surely. 

(iv) Let $\beta= \pr(\zeta_{t}^{O} \mbox{ survives})>0$, then, $\displaystyle{ I_{t} \Rightarrow (1-\beta) \delta_{\emptyset} + \beta \bar{\nu}_{\mu}}$.
\end{thm}
We comment on the proof of Theorem \ref{THEthm1}.The cornerstone for acquiring parts  \textit{(i)\textup{ and} (ii)} is to ascertain the existence of a sequence of space-time points, termed \textit{break points}, strictly increasing in both space and time, among which the behaviour of $r_{t}$ conditional on $\{\zeta_{t}^{O} \mbox{\textup{ survives}}\}$ stochastically replicates, these type of arguments have been established in Kuczek \cite{K}. Further, the proofs of parts \textit{(iii) \textup{and} (iv)} are based on variations of the arguments for the contact process case due to Durrett and Griffeath, see \cite{DG} and \cite{D80}, \cite{G}. We finally note that the technique of the proof for obtaining \textit{(i)} gives an alternative, elementary proof of the corresponding result for the contact process, see Theorem 1.4 in Durrett \cite{D80}. 

In the next section we introduce the graphical construction, we also present monotonicity and give some elementary coupling results. Section \ref{Sexp} is intended for the proof of two exponential estimates that we need for later. Section \ref{S31} is devoted to the study of break points, while in Section \ref{S32} we give the proof of Theorem \ref{THEthm1}.
%



\section{Preliminaries: The graphical construction}\label{grrep}
The graphical construction will be used in order to visualize the construction of various processes on the same probability space; we will repeatedly use it throughout this chapter.

Consider parameters $(\lambda,\mu)$ and, the other case being similar, suppose that $\mu \geq \lambda$.  To carry out our construction for all sites $x$ and $y=x-1,x+1$, let $(T_{n}^{x,y})_{n\geq1}$ and $(U_{n}^{x,y})_{n\geq1}$ be the event times of Poisson processes respectively at rates $\lambda$ and $\mu-\lambda$; further, let $(S_{n}^{x})_{n\geq 1}$ be the event times of a Poisson process at rate $1$. (All Poisson processes introduced are independent). 

Consider the space $\Z \times [0,\infty)$ thought of as giving a time line to each site of $\Z$; Cartesian product is denoted by $\times$. Given a realization of the before-mentioned ensemble of Poisson processes, we define the \textit{graphical construction} and $\zeta_{t}^{[\eta,s]}$,  $t\geq s$, three state contact process on $\Z$ with nearest neighbours interaction and parameters $(\lambda,\mu)$ started from $\eta$ at time $s\geq0$, i.e.\ $\zeta_{s}^{[\eta,s]} = \eta$, as follows. From each point $x \times T_{n}^{x,y}$ we place a directed $\lambda$-\textit{arrow} to $y \times T_{n}^{x,y}$; this indicates that at all times $t=T_{n}^{x,y}$, $t\geq s$, if $\zeta_{t-}^{[\eta,s]}(x)=1$ and $\zeta_{t-}^{[\eta,s]}(y)=0$ \textit{or} $\zeta_{t-}^{[\eta,s]}(y) = -1$ then we set $\zeta_{t}^{[\eta,s]}(y) = 1$ (where $\zeta_{t-}(x)$ denotes the limit of $\zeta_{t-\epsilon}(x)$ as $\epsilon\rightarrow 0$). From each point $x \times U_{n}^{x,y}$ we place a directed $(\mu-\lambda)$-\textit{arrow} to $y \times U_{n}^{x,y}$; this indicates that at any time $t=U_{n}^{x,y}$, $t\geq s$, if $\zeta_{t-}^{[\eta,s]}(x)=1$ and $\zeta_{t-}^{[\eta,s]}(y)=0$ then we set $\zeta_{t}^{[\eta,s]}(y) = 1$. While at each point $x \times S_{n}^{x}$ we place a \textit{recovery mark}; this indicates that at any time $t= S_{n}^{x}, t\geq s,$ if $\zeta_{t-}^{[\eta,s]}(x)=1$ then we set $\zeta_{t}^{[\eta,s]}(x) = 0$.  The reason we introduced the special marks is to make connection with percolation and hence the contact process, we define  the contact process $\xi_{t}^{A}$  with parameter $\mu$ started from $A \subset \Z$ as follows. We write $A \times 0 \rightarrow B \times t$, $t\geq 0$, if there exists a connected oriented path from $x \times 0$ to $y \times t$, for some $x \in A$ and $y \in B$, that moves along arrows (of either type) in the direction of the arrow and along time lines in increasing time direction without passing through a recovery mark, defining $\xi_{t}^{A} := \{x: A \times 0 \rightarrow x \times t\}$, $t\geq0$, we have that $(\xi_{t}^{A})$ is a set valued version of the contact process with parameter $\mu$ started from $A$ infected.

It is important to emphasize that the graphical construction, for fixed $(\lambda,\mu)$, defines all  $\zeta_{t}^{[\eta,s]}$,  $t\geq s$, for any configuration $\eta$ and time $s\geq0$, and all $\xi_{t}^{A}$, for any $A \subset \Z$, simultaneously on the same probability space, and hence provides a coupling of all these processes.

\begin{definition}\label{calI}
We shall denote by $\mathcal{I}(\zeta)$ the set of infected sites of any given configuration $\zeta$, that is, $\mathcal{I}(\zeta) = \{y \in \Z:\zeta(y) =1\}$.
\end{definition}

To simplify our notation, consistently to Section \ref{S0}, $\zeta_{t}^{[\eta,0]}$ is denoted as $\zeta_{t}^{\eta}$, and, letting $\eta_{0}$ be the standard initial configuration, $\zeta_{t}^{[\eta_{0},0]}$ is denoted as $\zeta_{t}^{O}$. Additionally, the event $\{\mathcal{I}(\zeta_{t}^{[\eta,s]}) \not= \emptyset  \mbox{ for all } t\geq s\}$ will be abbreviated below as  $\{\zeta_{t}^{[\eta,s]} \mbox{ survives}\}$. 

Finally, we note that we have produced a version of $\zeta_{t}^{\eta}$ via a countable collection of Poisson processes, this provides well-definedness of the process. Indeed, whenever one assumes that $|\mathcal{I}(\eta)|<\infty$, this is a consequence of standard Markov chains having an almost surely countable state space; otherwise, this is provided by an argument due to Harris \cite{H}, see Theorem 2.1 in Durrett \cite{D95}.

\section{Monotonicity, coupling results}\label{coupling}

To introduce monotonicity concepts, we endow the space of configurations $\{-1,0,1\}^{\Z}$ with the \textit{component-wise partial order}, $\mbox{i.e.}$, for any two configurations $\eta_{1}, \eta_{2}$ we have that $\eta_{1} \leq \eta_{2}$ whenever $\eta_{1}(x) \leq \eta_{2}(x)$ for all $x \in \Z$. We also note that for the conclusion of all statements in this section to hold the condition that $\mu \geq \lambda$ is necessary, this is straightforward to see. The following theorem is a known result, for a proof we refer to section 5 in Stacey \cite{S}, see also Theorem \ref{PIPmon} in the last chapter. 

\begin{thm}\label{moninit1}
Let $\eta$ and $\eta'$ be any two configurations such that $\eta\leq \eta'$. Consider the respective three state contact processes $\zeta_{t}^{\eta}$ and $\zeta_{t}^{\eta'}$ with the same parameters $(\lambda,\mu)$ coupled by the graphical construction. For all $(\lambda,\mu)$ such that $\mu \geq \lambda>0$, we have that $\zeta_{t}^{\eta} \leq \zeta_{t}^{\eta'}$ holds. We refer to this property as monotonicity in the initial configuration.
\end{thm}



For the remainder of this section we give various coupling results concerning $\zeta_{t}^{O}$, the three state contact process with parameters $(\lambda,\mu)$ started from the standard initial configuration, let $I_{t} = \mathcal{I}(\zeta_{t}^{O})$, $r_{t}= \sup I_{t}$ and $l_{t}= \inf I_{t}$. We note that the nearest neighbours assumption in all of the subsequent lemmas in this section is crucial.

The next lemma will be used repeatedly throughout this chapter, its proof given below is a simple extension of a well known result for the contact process on $\Z$ with nearest neighbours interaction, see  $\mbox{e.g.}$ $\cite{D80}$. 


\begin{lem}\label{piprendcoup1}
Let $\eta'$ be any configuration such that $\eta'(0)= 1$ and $\eta'(x)= -1$ for all $x\geq1$. Consider   $\zeta_{t}^{\eta'}$ with parameters $(\lambda,\mu)$ and let $r_{t}'= \sup\mathcal{I}(\zeta_{t}^{\eta'})$. For  $(\lambda,\mu)$ such that $\mu \geq \lambda$, if $\zeta_{t}^{O}$ and $\zeta_{t}^{\eta'}$ are coupled by the graphical construction then the following property holds, for all $t\geq 0$,
\begin{equation*}\label{couprend2}
r_{t} = r_{t}'  \mbox{ on } \{ I_{t} \not= \emptyset\}.
\end{equation*} 
\end{lem}

\begin{proof}
We prove the following stronger statement, for all $t\geq 0$, 
\begin{equation}\label{couprend1}
\zeta_{t}^{O}(x) = \zeta_{t}^{\eta'}(x)   	\mbox{ for all } x \geq l_{t}, \mbox{ on }  \{I_{t} \not= \emptyset\}.
\end{equation} 
For $t=0$ (\ref{couprend1}) holds, we show that all possible transitions preserve it.  An increase of $l_{t}$ (i.e., a recovery mark at $l_{t}\times t$) as well as any transition changing the state of any site $x$ such that $ x\geq l_{t}+1$  preserve $(\ref{couprend1})$.  It remains to examine transitions that decrease $l_{t}$, by monotonicity in the initial configuration we have that the possible pairs of $(\zeta_{t}^{O}(l_{t}-1), \zeta_{t}^{\eta'}(l_{t}-1))$ are the following $(-1,0),(-1,1),(0,0),(0,1)$. In the first pair case $(\ref{couprend1})$ is preserved because $\lambda$-arrows are used for transitions $-1 \rightarrow 1$ as well as $ 0 \rightarrow 1$, while in the three remaining cases this is obvious, the proof of $(\ref{couprend1})$ is thus complete. 
\end{proof}
The next lemma will be used in the proof of the two final parts of Theorem $\ref{THEthm1}$, its proof is a simple variant of that of Lemma $\ref{piprendcoup1}$ and is thus omitted.

\begin{lem}\label{cccoup}
Let $\xi_{t}^{\Z}$ be the contact process on $\Z$ with nearest neighbours interaction and parameter $\mu$ started from $\Z$. For $(\lambda,\mu)$ such that $\mu \geq \lambda>0$, if $\zeta_{t}^{O}$ and $\xi_{t}^{\Z}$ are coupled by the graphical construction the following property holds, for all $t\geq 0$,
\begin{equation*}
I_{t} =  \xi_{t}^{\Z} \cap [l_{t},r_{t}] \mbox{ \textup{on} } \{I_{t} \not= \emptyset\}.
\end{equation*}
\end{lem}




\begin{definition}\label{defetak}
For each integer $k$, let $\eta_{k}$ be the configuration such that $\eta_{k}(k) = 1$ and $\eta_{k}(y) =-1$ for all $y \not= k$.
\end{definition}

Our final coupling result will be used in the definition of break points in Section $\ref{S31}$. 
To state the lemma, define the stopping times $\tau_{k} = \inf\{t: r_{t}=k\}$, $k\geq1$, and let $R = \sup_{t\geq0}r_{t}$.

\begin{lem}\label{Sk}
Let $(\lambda,\mu)$ be such that $\mu \geq \lambda>0$ and consider the graphical construction. Consider also the processes $\zeta_{t}^{[\eta_{k},\tau_{k}]}$, $k\geq1$, started at times $\tau_{k}$ from $\eta_{k}$, as in Definition \ref{defetak}. Then, for all $\mbox{ }k=1,\dots,R$ the following property holds, 
\begin{equation*}\label{eqSk}
\zeta_{t}^{O} \geq \zeta_{t}^{[\eta_{k}, \tau_{k}]}, \mbox{ for all } t \geq \tau_{k}.
\end{equation*}
\end{lem}

\begin{proof}
We have that $\zeta_{\tau_{k}}^{O}(k) =1$, because $\eta_{k}$ is the least infectious configuration such that $\eta_{k}(k)=1$, we also have $\zeta^{O}_{\tau_{k}} \geq \eta_{k}$ $\mbox{ for all }k=1,\dots,R,$ by monotonicity in the initial configuration the proof is complete. 
\end{proof}

\section{Exponential estimates}\label{Sexp}

This section is intended for proving two exponential estimates for three state contact processes that will be needed in Section $\ref{S31}$. The method used is based on a renormalization result of Durrett and Schinazi $\cite{DS}$ that is an extension of the well-known work of Bezuidenhout and Grimmett \cite{BG}.

Subsequent developments require understanding of oriented site percolation.
Consider the set of sites, $\mathbb{L} =\{ (y,n) \in \Z^{2} :  n \geq0 \mbox{ and } y+n \mbox{ is even}\}.$
For each site $(y,n) \in \mathbb{L}$ we associate an independent Bernoulli random variable $w(y,n)\in \{0,1\}$ with parameter $p>0$; if $w(y,n)=1$ we say that $(y,n)$ is \textit{open}. We write $(x,m)\rightarrow(y,n)$ whenever there exists a sequence of open sites $(x,m) \equiv (y_{0}, m),\dots, (y_{n-m}, n)  \equiv (y,n)$ such that and $|y_{i} - y_{i-1}|=1$ for all $i=1,\dots, n-m$. Define $(A_{n})_{n\geq0}$ with parameter $p$ as $A_{n} = \{y:(0,0) \rightarrow (y,n)\}$. We write $\{A_{n} \textup{ survives}\}$ as an abbreviation for $\{\forall \hspace{0.5mm} n\geq1: A_{n} \not= \emptyset \}$. 




The next proposition is the renormalization result, it is a consequence of Theorem  4.3 in Durrett \cite{D95}, where the comparison assumptions there hold due to Proposition 4.8 of Durrett and Schinazi $\cite{DS}$. For stating it, given constants $L,T$, we define the set of configurations $Z_{y} = \{ \zeta :  |\mathcal{I}(\zeta) \cap [-L+2Ly,L+2Ly] | \geq L^{0.6} \}$, for all integers $y$.


\begin{prop}\label{couplDS}
Let $\eta$ be any configuration such that $\eta \in Z_{0}$, consider $\zeta_{t}^{\eta}$ with parameters $(\lambda, \mu)$ such that $\mu >\mu_{c}$ and $\lambda>0$. For all $p<1$ there exist constants $L,T$ such that $\zeta_{t}^{\eta}$ can be coupled to $A_{n}$ with parameter $p$ so that, 
\[
y \in A_{n} \hspace{1mm} \Rightarrow \hspace{1mm} \zeta_{nT}^{\eta} \in Z_{y}
\]
$(y,n) \in \mathbb{L}$. In particular the process survives. 
\end{prop}


The first of the exponential estimates that we need for Section $\ref{S31}$ is the following.

\begin{prop}\label{expbounds3scp}
Consider $\zeta_{t}^{O}$ with parameters $(\lambda, \mu)$. Let also $I_{t} = \mathcal{I}(\zeta_{t}^{O})$, $r_{t} = \sup I_{t}$ and $R= \sup_{t\geq0} r_{t}$, further let $\rho = \inf\{t: I_{t} = \emptyset\}$. If $(\lambda, \mu)$ are such that $\mu >\mu_{c}$ and $\mu\geq\lambda>0$ then there exist constants $C$ and $\gamma>0$ such that
\begin{equation}\label{eqexpbounds3scp}
\pr(R  \geq n, \mbox{ }\rho<\infty) \leq Ce^{-\gamma n}, 
\end{equation}
for all $n\geq1$.
\end{prop}

\begin{proof}
Consider the graphical construction for  $(\lambda, \mu)$ as in the statement. Recall the component-wise partial order on the space of configurations, the property of monotonicity in the initial configuration that were introduced in section $\ref{coupling}$ and,  the configurations $\eta_{k}$ as in Definition \ref{defetak}. By Proposition $\ref{couplDS}$, emulating the proof of Theorem 2.30 (a) of Liggett \cite{L99}, we have that 
\begin{equation}\label{rhoconfin}
\mbox{ } \pr( t < \rho < \infty) \leq Ce^{-\gamma t}, 
\end{equation}
for all $t\geq0$; to see that the argument in \cite{L99} applies in this context note that, by monotonicity in the initial configuration, for any time $s>0$ and any $x\in I_{s}$, considering the process $\zeta_{t}^{ [\eta_{x}, s]}$ we have that $\zeta_{t}^{O} \geq  \zeta_{t}^{[\eta_{x}, s]}$ for all $t\geq s$, hence, the proof we referred to applies for $\delta$ there taken to be $\delta = \pr(\zeta_{1}^{O} \in Z_{0}) >0$.  

For proving $(\ref{eqexpbounds3scp})$, by set theory we have that for all $n \geq 1$,
\begin{equation*}\label{RT}
\pr\left(R  > n, \rho < \infty\right) \leq \pr\left( \frac{n}{\lambda} < \rho< \infty\right) + \pr\left(\rho < \frac{n}{\lambda}, \mbox{ } R  > n\right)
\end{equation*}
the first term on the right hand side decays exponentially in $n$ due to $(\ref{rhoconfin})$, thus, it remains to prove that the probability of the event $\{\sup_{t\leq \frac{n}{\lambda}} r_{t} >n\}$ decays exponentially in $n$, which however is immediate because $\sup_{t \in (0,u]} r_{t}$ is bounded above in distribution by the number of events by time $u$ in a Poisson process at rate $\lambda$ and standard large deviations results for Poisson processes. 
\end{proof}

The other exponential estimate we will need in Section $\ref{S31}$ is the following.

\begin{prop}\label{shadldev2}
Let  $\bar{\eta}$ be such that $\bar{\eta}(x) =1$ for all $x \leq 0$ while $\bar{\eta}(x) =-1$ otherwise. Consider $\zeta^{\bar{\eta}}_{t}$ with parameters $(\lambda, \mu)$ and let $\displaystyle{ \bar{r}_{t} =\sup \mathcal{I}(\zeta^{\bar{\eta}}_{t})}$. If $(\lambda, \mu)$ are such that $\mu >\mu_{c}$ and $\mu\geq\lambda>0$ then there exist strictly positive and finite constants $a,\gamma$ and $C$ such that
\[
\pr\left(\bar{r}_{t}  < at\right) \leq Ce^{-\gamma t},
\]
for all $t\geq0$.
\end{prop}

\begin{proof} 
The next elementary result for independent site percolation as well as the subsequent easy geometrical lemma are used in the proof of Proposition $\ref{shadldev2}$, their proofs are given below for completeness. 

\begin{lem}\label{indperc}
Consider $(A_{n})$ with parameter $p$ and let $R_{n} = \sup A_{n}, n\geq0$.
For $p$ sufficiently close to 1 there are strictly positive and finite constants $a,\gamma$ and $C$ such that
\begin{equation*}\label{Rnspproc}
\pr( R_{n} < an, \mbox{ } A_{n} \textup{ survives}) \leq Ce^{- \gamma n}, 
\end{equation*}
for all $n\geq1$. 
\end{lem}


\begin{lem}\label{geomRR+}
Let $b,c$ be strictly positive constants such that $c<b$. For any $a < c$ we can choose sufficiently small $\phi \in (0,1)$, that does not depend on $t \in \R_{+}$, such that for all $x \in [-b\phi t,b\phi t]$, 
\begin{equation}\label{eq:geom}
[x - c (1-\phi)t, x + c (1-\phi)t] \supseteq [-at,at].
\end{equation}
\end{lem}

Consider the graphical construction for  $(\lambda, \mu)$ as in the statement. Let $p$ be sufficiently close to $1$ so that Lemma $\ref{indperc}$ is satisfied. Recall the configurations $\eta_{x}$ as in Definition \ref{defetak}. By the proof of Theorem 2.30 (a) of Liggett \cite{L99}---which applies for the reasons explained in the first paragraph of the proof of Proposition $\ref{expbounds3scp}$---, we have that total time $\sigma$ until we get a percolation process $A_{n}$ with parameter $p$ that is coupled to $\zeta_{t}^{[\eta_{\bar{r}_{\sigma}},\sigma]}$ as explained in Proposition $\ref{couplDS}$ (for $\bar{r}_{\sigma} \times (\sigma +1)$ being thought of as the origin) and is conditioned on $\{A_{n} \textup{ survives}\}$, is exponentially bounded. From this, because $\bar{r}_{t}$ is bounded above in distribution by a Poisson process,  we have that there exists a constant $\tilde\lambda$ such that the event 
$\left\{\bar{r}_{\sigma} \times (\sigma +1) \in  [ -\tilde{\lambda}t d, \tilde{\lambda}td] \times (0, td]\right\}$, for all $d \in (0,1)$, occurs outside some exponentially small probability in $t$. Finally on this event, by Lemma $\ref{indperc}$ and the coupling in Lemma $\ref{piprendcoup1}$, we have that there exists an $\tilde{a}>0$ such that $\bar{r}_{t} \geq \tilde{a}t - \bar{r}_{\sigma}$, again outside some exponentially small probability in $t$, choosing $\tilde{\lambda}=b$  and $\tilde{a}=c$ in Lemma $\ref{geomRR+}$ completes the proof.

\end{proof}

\begin{proof}[Proof of Lemma \ref{indperc}]
Define $A'_{n} = \{y: (x,0) \rightarrow (y,n) \mbox{ for some } x \leq0\}$ and let $R'_{n}= \sup A'_{n}$, $n\geq1$.
Because $R_{n} = R'_{n}$ on $\{A_{n} \textup{ survives}\}$, it is sufficient to prove that $p$ can be chosen sufficiently close to 1 such that, for some $a>0$, the probability of the event $R'_{n} < an$ decays exponentially in $n \geq 0$. Letting $B_{n}'$ be independent oriented bond percolation on $\mathbb{L}$ with supercritical parameter $\tilde{p}<1$ started from $\{(x,0) \in \mathbb{L} : x\leq0 \}$, the result follows from the corresponding large deviations result for $B_{n}'$ (see Durrett \cite{D84}, (1) in section 11), because for $p= \tilde{p}(2-\tilde{p})$ we have that $B_{n}'$ can be coupled to $A_{n}'$ such that $B_{n}' \subset A_{n}'$ holds, see Liggett \cite{L99}, p.13.
\end{proof}

\begin{proof}[Proof of Lemma \ref{geomRR+}]
Note that it is sufficient to consider $x = b \phi t$; then, simply choose $\phi$ such that $bt\phi -c(1-\phi)t< - at$, $\mbox{i.e.}$ for 
$\displaystyle{ \phi < \frac{c-a}{c+b}}$, $\phi>0$, equation $(\ref{eq:geom})$ holds. 
\end{proof}


\section{Break points}\label{S31} 

In this section we will prove Theorem \ref{THEprop} stated below; based solely on this theorem, we prove Theorem \ref{THEthm1} in Section \ref{S32}.

\begin{thm}\label{THEprop}
Consider $\zeta_{t}^{O}$ with parameters $(\lambda, \mu)$ and let $r_{t} = \sup\mathcal{I}(\zeta_{t}^{O})$. Suppose $(\lambda, \mu)$ such that $\mu >\mu_{c}$ and $\mu\geq\lambda>0$. On $\{\zeta^{O}_{t} \mbox{\textup{ survives}}\}$ there exist random (but not stopping) times $\tilde{\tau}_{0}:=0 < \tilde{\tau}_{1} < \tilde{\tau}_{2} < \dots$ such that $(r_{\tilde{\tau}_{n}} -r_{\tilde{\tau}_{n-1}}, \tilde{\tau}_{n}- \tilde{\tau}_{n-1})_{n \geq 1}$ are i.i.d.\ random vectors, where also $r_{\tilde{\tau}_{1}} \geq 1$ and $\displaystyle{r_{\tilde{\tau}_{n}} = \sup_{t\leq \tilde{\tau}_{n}} r_{t}}$. Furthermore, letting $M_{n}=r_{\tilde{\tau}_{n}}-\inf_{t \in[ \tilde{\tau}_{n}, \tilde{\tau}_{n+1})} r_{t}$, $n \geq 0$, we  have that $(M_{n})_{n\geq0}$ are i.i.d.\ random variables, where also $M_{n} \geq0$. Finally,  $r_{\tilde{\tau_{1}}}, \tilde{\tau}_{1},M_{0}$ are exponentially bounded. 
\end{thm}



%


For defining the break points below, consider the graphical construction for $(\lambda, \mu)$ such that $\mu >\mu_{c}$ and $\mu\geq\lambda>0$, consider $\zeta_{t}^{O}$ and define $r_{t} =\sup \mathcal{I}(\zeta_{t}^{O})$, define also the stopping times $\tau_{k} = \inf\{t: r_{t}=k\}$, $k\geq0$. Let also $\eta_{k}$ be as in Definition \ref{defetak}. The break points defined below is the unique strictly increasing, in space and in time, subsequence of the space-time points $k \times \tau_{k}, k\geq1,$ such that  $\zeta_{t}^{[\eta_{k},\tau_{k}]} \textup{ survives}$. The origin $0 \times 0$ is a break point, i.e.\ our subsequence is identified on $\{\zeta_{t}^{O} \textup{ survives}\}$.

\begin{definition}\label{definbpts}
Define $(K_{0}, \tau_{K_{0}})= (0, 0)$. For all $n\geq0$ and $K_{n}<\infty$ we inductively define
\[
K_{n+1} = \inf\{k \geq K_{n}+1:  \zeta_{t}^{[\eta_{k},\tau_{k}]} \mbox{ survives}\}, 
\]  
and $X_{n+1} = K_{n+1}-K_{n}$, additionally we define $\Psi_{n+1} = \tau_{K_{n+1}} - \tau_{K_{n}}$, and also
$\displaystyle{M_{n} =K_{n} -\inf_{\tau_{K_{n}} \leq t < \tau_{K_{n+1}}} r_{t}}$. We refer to the space-time points $K_{n} \times \tau_{K_{n}}$, $n\geq0$, as the \textit{break points}.
\end{definition}



Letting $\tilde{\tau}_{n}:= \tau_{K_{n}}, n\geq0$, in the definition above gives us that for proving Theorem $\ref{THEprop}$ it is sufficient to prove the two propositions following; this section is intended for proving these.

\begin{prop}\label{PROPexpbnd}
$K_{1}$, $\tau_{K_{1}}$ and $M_{0}$ are exponentially bounded.
\end{prop}


\begin{prop}\label{PROPiid}
$(X_{n}, \Psi_{n} ,M_{n-1})_{n \geq 1}$, are independent identically distributed vectors.
\end{prop} 

\begin{definition}
Given a configuration $\zeta$ and an integer $y\geq1$, define the configuration $\zeta-y$ to be $(\zeta-y)(x) = \zeta(y+x)$, for all $x \in \Z$.
\end{definition}

We shall denote by $\mathcal{F}_{t}$ the sigma algebra associated to the ensemble of Poisson processes used for producing the graphical construction up to time $t$.

The setting of the following lemma is important to what follows.

\begin{lem}\label{bptsinf}
Let  $\bar{\eta}$ such that $\bar{\eta}(x) =1$ for all $x \leq 0$ while $\bar{\eta}(x) =-1$ otherwise. Consider $\zeta^{\bar{\eta}}_{t}$ with parameters $(\lambda, \mu)$. Define $\bar{r}_{t} =\sup\mathcal{I}(\zeta^{\bar{\eta}}_{t})$, define also, the stopping times $T_{n} = \inf\{t: \bar{r}_{t} = n\}$, $n\geq0$. Let  $(\lambda,\mu)$ be such that $\mu \geq \lambda>0$ and $\mu > \mu_{c}$ and consider the graphical construction. 

Let $Y_{1}:=1$ and consider $\zeta_{t}^{1}:=\zeta_{t}^{[\eta_{Y_{1}}, T_{1}]}$, we let $\rho_{1} = \inf\{ t\geq T_{1}: \mathcal{I}(\zeta_{t}^{1}) = \emptyset\}$. 
For all $ n\geq 1$, proceed inductively: On the event $\{\rho_{n} < \infty\}$ let
\[
Y_{n+1} = 1+ \sup_{t \in [T_{Y_{n}},\rho_{n})} \bar{r}_{t}, 
\]
and consider $\zeta_{t}^{n+1}:= \zeta_{t}^{[\eta_{Y_{n+1}},T_{Y_{n+1}}]}$, we let $\rho_{n+1} = \inf\{t \geq T_{Y_{n+1}}: \mathcal{I}(\zeta_{t}^{n+1}) = \emptyset\}$; on the event that $\{\rho_{n} = \infty\}$ let $\rho_{l} = \infty$ for all $l > n$. Define the random variable $N = \inf\{n\geq1: \rho_{n}= \infty\}$. We have the following expression,
\begin{equation}\label{Yinf}
Y_{N} = \inf\{k\geq1: \zeta_{t}^{[\eta_{k}, T_{k}]} \mbox{ \textup{survives}}\},
\end{equation}
and also,
\begin{equation}\label{eq:algopiprendcoup}
\bar{r}_{t} = \sup\mathcal{I}(\zeta_{t}^{n}), \mbox{ for all } t \in [T_{Y_{n}},\rho_{n}) \mbox{ and } n\geq1.
\end{equation}

We further have that
\begin{equation}\label{cbpts1}
(\zeta^{1}_{t+ T_{1}}- 1)_{t \geq 0} \mbox{ is independent of }\mathcal{F}_{T_{1}}   \mbox{ and is equal in distribution to } (\zeta_{t}^{O})_{t\geq0},
\end{equation}
and also, 
\begin{eqnarray}\label{cbpts2}
&& \mbox{ conditional on } \{\rho_{n}<\infty, Y_{n+1} = w\}, w \geq 1, (\zeta^{n+1}_{t+ T_{Y_{n+1}}}- w)_{t \geq 0} \nonumber\\
&& \mbox{ is independent of } \mathcal{F}_{T_{Y_{n+1}}} \mbox{ and is equal in distribution to } (\zeta_{t}^{O})_{t\geq0}.
\end{eqnarray} 
\end{lem}

\begin{proof}

Equation $(\ref{Yinf})$ is a consequence of Lemma $\ref{Sk}$, to see this note that this lemma gives that for all $n\geq1$ on $\{\rho_{n}<\infty\}$, $\rho_{n} \geq \inf\{t\geq T_{k}: \mathcal{I}(\zeta_{t}^{[\eta_{k},T_{k}]})= \emptyset\}$ $\mbox{for all } k=Y_{n}+1,\dots,Y_{n+1}-1$. Equation $(\ref{eq:algopiprendcoup})$ is immediate due to Lemma $\ref{piprendcoup1}$.

Note that from Proposition $\ref{shadldev2}$ we have that $T_{n} < \infty \mbox{ for all }  n\geq0$ a.s.. Then, equation $(\ref{cbpts1})$ follows from the Strong Markov Property at time $T_{1}<\infty$ and translation invariance; while $(\ref{cbpts2})$ is also immediate by applying the Strong Markov Property at time $T_{Y_{n+1}}<\infty$, where $T_{Y_{n+1}}<\infty$ because from Proposition $\ref{expbounds3scp}$ we have that, conditional on $\rho_{n}<\infty$,  $Y_{n+1} < \infty$ $\mbox{a.s.}$.
\end{proof}

The connection between the break points and Lemma $\ref{bptsinf}$ comes by the following coupling result that is an immediate consequence of Lemma $\ref{piprendcoup1}$.
 
\begin{lem}\label{KeqK'}
Let $\eta'$ be any configuration such that $\eta'(0)= 1$ and $\eta'(x)= -1$ for all $x\geq1$. Consider   $\zeta_{t}^{\eta'}$ with parameters $(\lambda,\mu)$ and let $r_{t}'= \sup\mathcal{I}(\zeta_{t}^{\eta'})$, let also $\tau'_{k} = \inf\{t\geq 0: r'_{t} = k\}$, $k\geq1$. Define the integers 
\[
K' = \inf\{k\geq 1: \zeta_{t}^{[\eta_{k}, \tau'_{k}]} \mbox{\textup{ survives}}\}, 
\]
and also $M'= \inf_{0 \leq t \leq \tau_{K}'} r_{t}'.$ Consider further $\zeta_{t}^{O}$  with parameters $(\lambda,\mu)$. For $(\lambda,\mu)$ such that $\mu \geq \lambda>0$ and $\mu >\mu_{c}$, if $\zeta_{t}^{O}$ and $\zeta_{t}^{\eta'}$ are coupled by the graphical construction the following property holds,
\begin{equation*}\label{coupK1tauK1}
 (K', \tau'_{K'},M') = (K_{1}, \tau_{K_{1}}, M_{0}), \mbox{ on  } \{\zeta_{t}^{O} \mbox{\textup{ survives}}\},
\end{equation*} 
where $K_{1}, \tau_{K_{1}}, M_{0}$ are as in Definition $\ref{definbpts}$. 
\end{lem}



\begin{proof}[proof of Proposition $\ref{PROPexpbnd}$]
Consider the setting of Lemma $\ref{bptsinf}$. By the definition of break points, Definition \ref{definbpts}, and Lemma $\ref{KeqK'}$ we have that on $\{\zeta_{t}^{O} \mbox{ survives}\}$, $K_{1}=Y_{N}$, $\tau_{K_{1}}=  T_{Y_{N}}$ and $M_{0}=\inf_{t \leq T_{Y_{N}}}\bar{r}_{t}$. It is thus sufficient to prove that the random variables $Y_{N}, T_{Y_{N}}, \inf_{t \leq T_{Y_{N}}}\bar{r}_{t}$ are exponentially bounded, merely because an exponentially bounded random variable is again exponentially bounded conditional on any set of positive probability.

We have that 
\begin{equation}\label{eq:YN}
Y_{N} = 1+ \sum\limits_{n=2}^{N} (Y_{k} - Y_{k-1}) \mbox{ on } \{N\geq2\},
\end{equation} 
while $Y_{1}:=1$, using this and Proposition \ref{expbounds3scp}, we will prove that $Y_{N}$ is bounded above in distribution by a geometric sum of $\mbox{i.i.d.}$ exponentially bounded random variables and hence is itself exponentially bounded. 

Let $\rho$ and $R$ be as in Proposition \ref{expbounds3scp}, we define $p_{R}(w)= \pr(R +1 = w, \rho<\infty)$, and $\bar{p}_{R}(w)= \pr(R +1= w | \mbox{ }\rho <\infty)$, for all integers $w\geq1$, define also $p= \pr(\rho = \infty)>0$ and $q=1-p$, where $p>0$ by Proposition \ref{couplDS}.

By $(\ref{cbpts1})$ of the statement of Lemma $\ref{bptsinf}$, we have that
\begin{equation}\label{eq:Y2}
\pr( Y_{2} -Y_{1} = w, \rho_{1} < \infty) = p_{R}(w) 
\end{equation}
$w\geq1$; similarly, from $(\ref{cbpts2})$ of the same statement, we have that, for all $n\geq1$,  
\begin{equation}\label{eq:rec1}
\pr(\rho_{n+1} = \infty| \mbox{ } \rho_{n} <\infty, Y_{n+1} = z, \mathcal{F}_{T_{Y_{n+1}}}) = p, 
\end{equation}
and also, 
\begin{equation}\label{eq:rec2}
\pr(Y_{n+1} - Y_{n} = w, \rho_{n}<\infty |\mbox{ } \rho_{n-1} <\infty, Y_{n}=z,\mathcal{F}_{T_{Y_{n}}}) = p_{R}(w) ,
\end{equation}
for all $w,z\geq1$.

However, $\{N=n\} = \{\rho_{k}<\infty \mbox{ for all } k =1,\dots,n-1 \mbox{ and } \rho_{n} = \infty\}$, $n\geq2$, 
and hence, 
$$
\displaylines{
\left\{\textstyle{\bigcap \limits_{n=1}^{m}} \{Y_{n+1}-Y_{n} = w_{n}\}, N=m+1\right\} = \hfill \cr
=  \left\{\textstyle{ \bigcap\limits_{n=1}^{m}} \{Y_{n+1}-Y_{n} = w_{n}, \rho_{n}<\infty\}, \rho_{m+1}=\infty\right\},}
$$
 for all $m\geq1$, using this, from $(\ref{eq:rec1})$, $(m-1)$ applications of $(\ref{eq:rec2})$, and $(\ref{eq:Y2})$, since $p_{R}(w)= q\bar{p}_{R}(w)$, we have that
\begin{equation*}\label{eq:prod-ind}
\pr\left(\bigcap \limits_{n=1}^{m} \{Y_{n+1}-Y_{n} = w_{n}\}, N=m+1\right)= p q^{m} \prod \limits_{n=1}^{m} \bar{p}_{R}(w_{n}),  
\end{equation*}
for all $m\geq1$ and $w_{n}\geq1$.  From the last display and $(\ref{eq:YN})$, due to Proposition \ref{expbounds3scp}, we have that $Y_{N}$ is exponentially bounded by an elementary conditioning argument as follows. Letting $(\tilde{\rho}_{k}, \tilde{R}_{k}), k\geq1$ be independent pairs of random variables each of which is distributed as $(\rho,R)$ and the geometric random variable $\tilde{N} := \inf\{n\geq1: \tilde{\rho_{n}} = \infty\}$, we have that $Y_{N}$ is equal in distribution to $\sum \limits_{k=0}^{\tilde{N}-1} \tilde{R_{k}}$, $\tilde{R}_{0}:=1$.


We proceed to prove that $T_{Y_{N}}$ and $\inf_{t \leq T_{Y_{N}}}\bar{r}_{t}$ are exponentially bounded random variables. By $(\ref{Yinf})$, letting $\bar{x}_{t} = \sup_{s \leq t}\bar{r}_{s}$, we have that $\{T_{Y_{N}} > t\}= \{\bar{x}_{t} \leq Y_{N}\}$; from this and set theory we have that, for any $a>0$
\begin{eqnarray}\label{bpexm}
\pr(T_{Y_{N}} > t) &=& \pr(\bar{x}_{t} \leq Y_{N}) \nonumber\\
&\leq& \pr(\bar{x}_{t} < at) + \pr(\bar{x}_{t}\geq at, \bar{x}_{t} \leq Y_{N}) \nonumber\\
&\leq& \pr(\bar{x}_{t} < at) + \pr(Y_{N} \geq \lfloor at \rfloor),
\end{eqnarray}
for all $t\geq0$, where $\lfloor \cdot\rfloor$ is the floor function; choosing $a>0$ as in Proposition $\ref{shadldev2}$, because $\bar{x}_{t} \geq \bar{r}_{t}$, and since $Y_{N}$ is exponentially bounded, we deduce by (\ref{bpexm}) that  $T_{Y_{N}}$ is exponentially bounded as well.

Finally, we prove that $M:= \inf_{t \leq T_{Y_{N}}}\bar{r}_{t}$ is exponentially bounded. 
From set theory, 
\[
\pr(M < -x) \leq \pr \left(T_{Y_{N}} \geq \frac{x}{\mu} \right) + \pr\left(T_{Y_{N}} < \frac{x}{\mu}, \{ \bar{r}_{s} \leq -x \mbox{ for some } s \leq T_{Y_{N}}\}\right), 
\]
because $T_{Y_{N}}$ is exponentially bounded, it is sufficient to prove that the second term of the right hand side decays exponentially. However, recall that $\bar{r}_{_{T_{Y_{N}}}} \geq 1$, hence, 
$$\displaylines{ \pr\left(T_{Y_{N}} < \frac{x}{\mu}, \{ \bar{r}_{s} \leq -x \mbox{ for some } s \leq T_{Y_{N}}\}\right) \leq \hfill \cr 
\leq \pr\left((\bar{r}_{t}- \bar{r}_{s}) > x \mbox{ for some } s  \leq \frac{x} {\mu} \mbox{ and } t  \leq \frac{x} {\mu} \right),}
$$
where the term on the right of the last display decays exponentially in $x$, because $(\bar{r}_{t} - \bar{r}_{s})$, $t >s$ is bounded above in distribution by $\Lambda_{\mu}(s,t]$, the number of events of a Poisson process at rate $\mu$ within the time interval $(s,t]$, by use of standard large deviations for Poisson processes, because $\Lambda_{\mu}(s,t] \leq \Lambda_{\mu}( 0,x / \mu]$ for any $s,t \in (0,x / \mu]$. 
\end{proof}

The next lemma is used in the proof of Proposition $\ref{PROPiid}$ following.
\begin{lem}\label{Xbdownii}
Consider the setting of the definition of break points, Definition \ref{definbpts}. For all $n\geq1$, we have that
\begin{equation}\label{indXn}
\big\{ \textstyle{ \bigcap \limits_{l=1}^{n} } \{(X_{l}, \Psi_{l}, M_{l-1}) = (x_{l}, t_{l}, m_{l-1}) \}, \zeta_{t}^{O} \textup{ survives}\big\} = \{ \zeta_{t}^{[\eta_{z_{n}},w_{n}]}\textup{ survives}, \tau_{z_{n}} = w_{n}, A \},
\end{equation}
for some event $A \in \mathcal{F}_{w_{n}}$, where $z_{n} = \sum\limits_{l=1}^{n} x_{l}$ and $w_{n} = \sum\limits_{l=1}^{n} t_{l}$.
\end{lem}

\begin{proof}
Considering the setting of Lemma $\ref{bptsinf}$ we have that 
\[
\{(Y_{N},T_{Y_{N}}, \inf_{t \leq T_{Y_{N}}}\bar{r}_{t}) = (x_{1}, t_{1}, m_{0})\} = \{\zeta_{t}^{[\eta_{x_{1}}, t_{1}]} \mbox{ survives}, T_{x_{1}} = t_{1}, B\},
\] 
for $B \in \mathcal{F}_{t_{1}}$; from this and Lemma $\ref{KeqK'}$ we have that 
$$\displaylines{  \{ (X_{1}, \Psi_{1}, M_{0}) = (x_{1}, t_{1}, m_{0}) ,\zeta_{t}^{O} \mbox{survives}\} \hfill \cr  
\hspace{7mm}=\{\zeta_{t}^{[\eta_{x_{1}},t_{1}]} \mbox{ survives},\tau_{x_{1}} = t_{1}, B, \zeta_{t}^{O} \mbox{survives}\} \cr 
= \{\zeta_{t}^{[\eta_{x_{1}},t_{1}]} \mbox{ survives},\tau_{x_{1}} = t_{1}, B, I_{t_{1}} \not=\emptyset\}
}
$$
for all $x_{1}\geq 1$, $t_{1} \in \R_{+}$, $m_{0} \geq0$, because $\{ I_{t_{1}} \not=\emptyset \} \in \mathcal{F}_{t_{1}}$ we have thus proved $(\ref{indXn})$ for $n=1$, by repeated applications of the last display the proof for general $n\geq1$ is derived. 
\end{proof}


\begin{proof}[proof of Proposition $\ref{PROPiid}$]
Consider the setting of the definition of break points, Definition \ref{definbpts}.
Assume that $K_{n}$, $\tau_{K_{n}}$, $M_{n-1}$ are almost surely finite, we will prove that
\begin{eqnarray}\label{Xn}
\pr\left((X_{n+1},\Psi_{n+1},M_{n}) = (x, t,m)\vline\mbox{ } \textstyle{ \bigcap \limits_{l=1}^{n}} \{(X_{l},\Psi_{l},M_{l-1}) =(x_{l}, t_{l},m_{l-1})\}, \zeta_{t}^{O} \mbox{survives}\right)&& \nonumber\\
=\mbox{ } \pr\big( (X_{1},\Psi_{1},M_{0})= (x, t,m)|\mbox{ } \zeta_{t}^{O} \mbox{survives} \big)\hspace{10mm}
\end{eqnarray}
for all $(x_{l},t_{l},m_{l-1})$, $x_{l} \geq 1, t_{l}\in \R_{+}$, $m_{l-1}\geq0$, $l=1,\dots,n$, and hence in particular that $K_{n+1}$, $\tau_{K_{n+1}}$, $M_{n}$ are exponentially bounded.  By induction because $K_{1}$ and $\tau_{K_{1}},M_{0}$ are exponentially bounded by Proposition $\ref{PROPexpbnd}$ we have that $(\ref{Xn})$ completes the proof of Proposition $\ref{PROPiid}$ by Bayes's sequential formula. 

It remains to prove $(\ref{Xn})$, rewrite the conditioning event in its left hand side according to $(\ref{indXn})$ in Lemma $\ref{Xbdownii}$ and note that 
\[
\{\tau_{z_{n}} = w_{n}\} \subset \{ \zeta_{w_{n}}^{O}(z_{n})=1 \mbox{ and } \zeta_{w_{n}}^{O}(y)= -1, \mbox{ for all } y \geq z_{n}+1\},
\]
thus, applying Lemma $\ref{KeqK'}$, gives the proof by independence of the Poisson processes in disjoint parts of the graphical construction, because $(\zeta_{t+w_{n}}^{[\eta_{z_{n}}, w_{n}]}-z_{n})_{t\geq0}$ is equal in distribution to $(\zeta_{t}^{O})_{t\geq0}$ by translation invariance.
\end{proof}



\section{Proof of Theorem $\ref{THEthm1}$}\label{S32}
We denote by $\bar{\pr}$ the probability measure induced by the construction of the process conditional on $\{\zeta_{t}^{O} \mbox{ survives}\}$ and, by $\bar{\E}$ the expectation associated to $\bar{\pr}$. Consider the setting of Theorem $\ref{THEprop}$ and let $\displaystyle{ \alpha =  \frac { \bar{\E} (r_{\tilde{\tau}_{1}})} { \bar{\E} (\tilde{\tau}_{1})}}$, $\alpha \in (0, \infty)$. 

\begin{proof}[proof of (i)]
Because $r_{\tilde{\tau}_{n}} = \sum\limits_{m = 1}^{n} (r_{\tilde{\tau}_{m}} - r_{\tilde{\tau}_{m-1}})$ and $\tilde{\tau}_{n} = \sum\limits_{m=1}^{n}( \tilde{\tau}_{m} - \tilde{\tau}_{m-1})$, $n\geq1$, using the strong law of large numbers twice gives us that
\begin{equation}\label{spdatbpts} 
\bar{\pr} \left( \lim_{n \rightarrow \infty} \frac{r_{\tilde{\tau}_{n}}}{\tilde{\tau}_{n}} = \alpha \right)=1,
\end{equation} 
we prove that indeed $\displaystyle{\lim_{t\rightarrow \infty}\frac{r_{t}}{t} =\alpha}$, $\bar{\pr}$ a.s..
From Theorem $\ref{THEprop}$ we have that
\begin{equation}\label{intt}
\frac{r_{\tilde{\tau}_{n}}-M_{n}}{\tilde{\tau}_{n+1}} \leq \frac{r_{t}}{t} \leq \frac{r_{\tilde{\tau}_{n+1}}} {\tilde{\tau}_{n}} , \mbox{ for all }  t\in[\tilde{\tau}_{n}, \tilde{\tau}_{n+1}), 
\end{equation}
$n\geq0$. Further, because $(M_{n})_{n\geq0}$, $M_{0}\geq0$, is a sequence of i.i.d.\ and exponentially bounded random variables we have that
\begin{equation}\label{spdM}
\bar{\pr}  \left(\lim_{n \rightarrow \infty} \frac{M_{n}}{n} =  0 \right) =1,
\end{equation}
by the 1st Borel-Cantelli lemma. Consider any $ a < \alpha$, by $(\ref{intt})$ we have that
\begin{equation}\label{eq:spdlessa}
\left\{ \frac{r_{t_{k}}}{t_{k}} < a  \mbox{ for some } t_{k} \uparrow \infty \right\}\subseteq \left\{ \limsup_{n\rightarrow \infty} \left\{\frac{r_{\tilde{\tau}_{n}}-M_{n}}{\tilde{\tau}_{n+1}} <  a \right\}\right\}, 
\end{equation}
however $\displaystyle{ \bar{\pr}\left(\limsup_{n\rightarrow \infty} \left\{\frac{r_{\tilde{\tau}_{n}}-M_{n}}{\tilde{\tau}_{n+1}} <  a \right\}\right)=0}$, to see this simply use $(\ref{spdM})$ and $(\ref{spdatbpts})$ to deduce that $\displaystyle{\lim_{n\rightarrow \infty}\frac{r_{\tilde{\tau}_{n}}-M_{n}}{\tilde{\tau}_{n+1}} = \alpha}$, $\bar{\pr}$ $\mbox{a.s.}$. By use of the upper bound in $(\ref{intt})$ and $(\ref{spdatbpts})$, we also have that for any $a > \alpha$, $\displaystyle{\bar{\pr}\left( \left\{ \frac{r_{t_{k}}}{t_{k}} > a  \mbox{ for some } t_{k} \uparrow \infty \right\}\right) =0}$, this completes the proof of \textit{(i)}. 
\end{proof}

\begin{proof}[proof of (ii)] We will prove that
\begin{equation*}
\lim_{t\rightarrow \infty}\bar{\pr} \left( \frac{r_{t} - \alpha t}{\sqrt{t}} \leq x \right) = \Phi\left(\frac{x}{\sigma^{2}}\right),
\end{equation*}
for some $\sigma^{2}>0$, $x \in\R$,  where $\Phi$ is the standard normal distribution function, $\mbox{i.e.}$,  $\displaystyle{ \Phi(y) := \frac{1}{\sqrt{2\pi}} \int_{-\infty}^{y}\exp\left(-\frac{1}{2} z^{2}\right)dz}$, $y\in \R$.

Define $N_{t}= \sup\{n: \tilde{\tau}_{n} <t\}$; evoking Lemma 2 in Kuczek \cite{K}, $\mbox{p.}$ 1330--1331, which applies due to Theorem $\ref{THEprop}$, we have that 
\begin{equation*}\label{kucinitNt}
\lim_{t\rightarrow \infty}\bar{\pr} \left( \frac{r_{N_{t}} - \alpha t}{\sqrt{t}} \leq x \right) = \Phi\left(\frac{x}{\sigma^{2}}\right),
\end{equation*}
$x \in\R$. From this, by standard association of convergence concepts, known as Slutsky's theorem, it is sufficient to show that
\begin{equation}\label{clt0}
\bar{\pr} \left(\lim_{t \rightarrow \infty } \frac{r_{t} - r_{N_{t}}}{\sqrt{t}}=0 \right)=1,
\end{equation}
and that $\sigma^{2}$ is strictly positive. Note however that, by Theorem $\ref{THEprop}$ we have that,
\begin{equation}\label{cltbouds}
\frac{ M_{\tilde{N}_{t}}} {\sqrt{t}} \leq \frac{ r_{t} - r_{N_{t}}}{\sqrt{t}} \leq \frac{ r_{\tilde{\tau}_{N_{t}+1}} - r_{\tilde{\tau}_{N_{t}}}}{\sqrt{t}} 
\end{equation}
for all $t\geq0$. 

We show that $(\ref{clt0})$ follows from $(\ref{cltbouds})$. Because $(r_{\tilde{\tau}_{n+1}} - r_{\tilde{\tau}_{n}})_{n\geq0}$, $r_{\tilde{\tau}_{1}}\geq1$, are $\mbox{i.i.d.}$ and exponentially bounded, by the 1st Borel-Cantelli lemma, and then the strong law of large numbers, we have that 
\[
\lim_{n\rightarrow \infty} \frac{ \frac{1}{\sqrt{n}}(r_{\tilde{\tau}_{n+1}} - r_{\tilde{\tau}_{n}})}{\sqrt{ \frac{\tilde{\tau}_{n}}{n}}}= 0
\] 
$\bar{\pr}$ $\mbox{a.s.}$, from the last display and emulating the argument given in $(\ref{eq:spdlessa})$ we have that 
$\displaystyle{ \lim_{t\rightarrow \infty} \frac{ r_{\tilde{\tau}_{N_{t}+1}} - r_{\tilde{\tau}_{N_{t}}}}{\sqrt{t}} =0}$, $\bar{\pr}$ $\mbox{a.s.}$.  Similarly, because $(M_{n})_{n\geq0}$ are non-negative $\mbox{i.i.d.}$ and exponentially bounded random variables, we also have that 
$\displaystyle{\lim_{t\rightarrow \infty}\frac{ M_{\tilde{N}_{t}}}{\sqrt{t}} =0}$, $\bar{\pr} \mbox{ a.s.}$. 

Finally, we show that $\sigma^{2}>0$. As in the proof of Corollary 1 in Kuczek \cite{K}, because $\displaystyle{ \alpha = \frac { \bar{\E} (r_{\tilde{\tau}_{1}})} { \bar{\E} (\tilde{\tau}_{1})}}$, we need to show that 
$\bar{\E}\left( r_{\tilde{\tau}_{1}}\bar{\E}(\tilde{\tau}_{1})  - \tilde{\tau}_{1} \bar{\E} (r_{\tilde{\tau}_{1}})\right)^{2}>0$. However, because  $r_{\tilde{\tau}_{1}} \geq 1$, this follows by Chebyshev's inequality. This completes the proof of \textit{(ii)}. 

\end{proof}



For the remainder of the proof consider the graphical construction for $(\lambda, \mu)$ such that $\mu >\mu_{c}$ and $\mu\geq\lambda>0$. Consider $\zeta_{t}^{O}$, let $r_{t} =\sup I_{t}$ and $l_{t} = \inf I_{t}$ be  respectively the rightmost and leftmost infected of $I_{t}= \mathcal{I}(\zeta_{t}^{O})$. Consider also $\xi_{t}^{\Z}$, the contact process with parameter $\mu$ started from $\Z$. By Lemma $\ref{cccoup}$ we have that, for all $t\geq0$, 
\begin{equation}\label{coupHtZ}
I_{t} =  \xi_{t}^{\Z} \cap [l_{t},r_{t}] \mbox{ \textup{on} } \{I_{t} \not= \emptyset\}.
\end{equation}

\begin{proof}[proof of (iii)]
Let $\theta= \theta(\mu) >0$ be the density of the upper invariant measure, that is, $\displaystyle{ \theta= \lim_{t\rightarrow \infty}\pr( x \in \xi_{t}^{\Z})}$. 
We prove that $\displaystyle{\lim_{t\rightarrow \infty}\frac{|I_{t}|}{t} = 2 \alpha \theta,}$ $\bar{\pr}$ a.s.. 

Considering the interval $[\max\{l_{t}, -\alpha t\}, \min\{r_{t},\alpha t\}]$, we have that for all $t\geq0$,
\begin{equation}\label{sllnineq}
\vline \mbox{ } \sum_{x=l_{t}}^{r_{t}} 1( x \in \xi_{t}^{\Z}) - \sum_{x =- \alpha t}^{ \alpha t} 1(x \in \xi_{t}^{\Z}) \mbox{ }\vline \leq | r_{t} - \alpha t| + | l_{t} + \alpha t|, \mbox{ on }\{I_{t} \not= \emptyset \},
\end{equation}
where $1(\cdot)$ denotes the indicator function. However, by $(\ref{coupHtZ})$,  $\displaystyle{ |I_{t}|= \sum_{x=l_{t}}^{r_{t}} 1(x \in \xi_{t}^{\Z})}$,  $\mbox{on } \{I_{t} \not= \emptyset \}$,  thus, because $\displaystyle{\lim_{t\rightarrow \infty} \frac{r_{t}}{t} = \alpha}$ and, by symmetry, $\displaystyle{  \lim_{t\rightarrow \infty} \frac{l_{t}}{t} =-\alpha}$, $\bar{\pr}$ a.s., the proof follows from $(\ref{sllnineq})$ because, for any $a>0$, $\displaystyle{ \lim_{t\rightarrow \infty} \frac{1}{t}\sum_{|x| \leq a t} 1(x \in \xi_{t}^{\Z}) = 2 a \theta}$, $\pr$ a.s., for a proof we refer see equation (19) in the proof of Theorem 9 of Durrett and Griffeath \cite{DG}.
\end{proof}

\begin{proof}[proof of (iv)] Let $\rho = \inf\{t\geq0: I_{t} = \emptyset\}$. In the context of set valued processes, by general considerations, see Durrett \cite{D95}, it is known that weak convergence is equivalent to convergence of finite dimensional distributions and that, by inclusion-exclusion,  it is equivalent to show that for any finite set of sites $F\subset \Z$ 
\begin{equation*}\label{eq:compconv}
\lim_{t\rightarrow \infty} \pr(I_{t} \cap F = \emptyset) = \pr(\rho < \infty) + \pr(\rho = \infty)\phi_{F}(\emptyset),
\end{equation*}
where $\displaystyle{ \phi_{F}(\emptyset):= \lim_{t\rightarrow \infty} \pr(\xi_{t}^{\Z} \cap F =\emptyset)}$. By set theory, it is sufficient to prove that
$\displaystyle{ \lim_{t\rightarrow \infty} \pr(I_{t} \cap F = \emptyset, \rho \geq t)= \pr(\rho = \infty)\phi_{F}(\emptyset)}$,  because $\{\rho <t\} \subseteq \{I_{t}\cap F =\emptyset\}$. However, emulating the proof of the respective result for the contact process (see $\mbox{e.g.}$ Theorem 5.1 in Griffeath \cite{G}), we have that $\displaystyle{\lim_{t \rightarrow \infty} \pr(\xi_{t}^{\Z} \cap F = \emptyset, \rho \geq t) = \pr(\rho =\infty)\phi_{F}(\emptyset)}$, hence, it is sufficient to prove that 
\begin{equation}\label{cccoupconseq}
\limsup_{t \rightarrow \infty} \pr(I_{t} \cap F = \emptyset, \rho \geq t) \leq  \lim_{t\rightarrow \infty}\pr(\xi_{t}^{\Z} \cap F = \emptyset, \rho \geq t),
\end{equation}
because also $\{ I_{t} \cap F = \emptyset, \rho \geq t\} \supseteq \{ \xi_{t}^{\Z} \cap F = \emptyset, \rho \geq t\}$,  by $(\ref{coupHtZ})$. 

It remains to prove $(\ref{cccoupconseq})$. By elementary calculations, 
\begin{equation*}\label{ccrhoinf}
\pr(I_{t} \cap F = \emptyset, \rho = \infty) - \pr(\xi_{t}^{\Z} \cap F = \emptyset, \rho \geq t) \leq \pr(  \xi_{t}^{\Z} \cap F \supsetneq I_{t} \cap F, \rho = \infty), 
\end{equation*}
for all $t\geq0$, where we used that by $(\ref{coupHtZ})$, $I_{t} \subset \xi_{t}^{\Z}$ for all $t\geq0$. From the last display above and set theory we have that 
$$\displaylines{\pr(I_{t} \cap F = \emptyset, \rho \geq t)-  \pr(\xi_{t}^{\Z} \cap F = \emptyset, \rho \geq t) \hfill \cr
\hfill \leq \pr( \xi_{t}^{\Z} \cap F \supsetneq I_{t} \cap F, \rho = \infty) + \pr( t<\rho <\infty),  }$$
for all $t\geq0$, however the limit as $t\rightarrow \infty$ of both terms of the right hand side in the above display is $0$, for the former this comes by $(\ref{coupHtZ})$, because $\displaystyle{\lim_{t\rightarrow \infty} r_{t} = \infty}$ and $\displaystyle{\lim_{t\rightarrow \infty} l_{t} = \infty}$, $\bar{\pr}$ $\mbox{a.s.}$, while for the latter this is obvious.


\end{proof}

\chapter{Convergence rates} 

\paragraph{Abstract:} The contact process on the integers with nearest neighbours interaction and infection rate $\mu$ altered so that initial infections occur at rate $\lambda$ instead is considered. It is known that regardless of the value of $\lambda$, if $\mu$ is less than the contact process's critical value then the process dies out, while if $\mu$ is greater than that value then the process survives. In the former case the time to die out is shown to be exponentially bounded; in the latter case, assuming additionally that $\mu\geq\lambda$, the ratio of the limit of the speed to that of the unaltered contact process is shown to be at most $\lambda / \mu$.
\vfill



%
\section{Introduction and main results}\label{S0}
This chapter is concerned with the three state contact process on the integers with nearest neighbours interaction and parameters $(\lambda,\mu)$, which is briefly described as follows. The collection of the states of the sites at time $t$ is denoted by $\zeta_{t}= \{\zeta_{t}(x), x\in \Z\}$ and is referred to as the configuration of the process.  The site $x$ at time $t$ is regarded   as infected if $\zeta_{t}(x) = 1$, as susceptible and never infected if $\zeta_{t}(x) = -1$ and, as susceptible and previously infected if $\zeta_{t}(x) = 0$. The dynamics for the evolution of $\zeta_{t}$ are specified locally. Transitions of $\zeta_{t}(x)$ occur according to the rules:  $1 \rightarrow 0$ at rate 1, $-1 \rightarrow 1$  at rate $\lambda \hspace{0.2mm} n(x)$ and $0 \rightarrow 1$  at rate $\mu \hspace{0.2mm}n(x)$, where $n(x)$ is the number of nearest neighbours of $x$ that are infected in $\zeta_{t}$. For a formal definition of the continuous time Markov process $\zeta_{t}$ on $\{-1,0,1\}^{\Z}$ we refer the reader to \cite{D95}, \cite{L}.


We shall use $\zeta_{t}^{O}$ to denote the process started from the origin infected and all other sites susceptible and never infected, this configuration is referred to as the \textit{standard initial configuration}. In general however,  $\zeta_{t}^{\eta}$ will denote the process started from configuration $\eta$. The process is said to \textit{survive} 
if $(\lambda,\mu)$ are such that $\pr(\zeta_{t}^{O} \mbox{ survives})>0$, where the event $\{\zeta_{t} \mbox{ survives}\}$ is an abbreviation for $\{\forall \hspace{0.5mm} t\geq0, \hspace{1mm} \zeta_{t}(x) =1 \text{ for some } x \}$. When the process does not survive it is said to \textit{die out}.  
 
When $(\lambda,\mu)$ are such that $\lambda = \mu$ the process is reduced to the well known contact process. It is well known that the contact process exhibits a phase transition phenomenon; that is, letting $\mu_{c}$ denote its critical value on the integers with nearest neighbours interaction,  $0<\mu_{c}<\infty$. For an account of various related results and proofs we refer the reader to \cite{L}, \cite{D88} and \cite{L99}. We also note that for this process we will identify a configuration with the subset of $\Z$ that corresponds to the set of its infected sites, since states $-1$ and $0$ are effectively equivalent.


It is shown in Durrett and Schinazi \cite{DS} that for $(\lambda,\mu)$ such that $\mu< \mu_{c}$ and $\lambda<\infty$ the process dies out. Taking our own approach we extend this result by giving the following exponential bounds for the range and the duration of the epidemic. 

\begin{thm}\label{thmsub}
For $(\lambda,\mu)$ such that $\mu<\mu_{c}$ and $\lambda<\infty$, there exists $\delta<1$ such that
$\pr\left(\exists \hspace{0.5mm} t \textup{ s.t., } \zeta_{t}^{O}(n)=1 \textup{ or } \zeta_{t}^{O}(-n)=1 \right) \leq \delta^{n}, \mbox{ for all } n\geq1$; further, there exist $C$ and $\gamma>0$ such that
$\pr\left(\zeta_{t}^{O}(x) =1 \textup{ for some } x \right) \leq C e^{-\gamma t}, \mbox{ for all } t\geq0.$
\end{thm}

The method of proof of the first part of Theorem \ref{thmsub} is based on deduction to the contact process, for which we prove that the probability that the infection never escapes an interval of infected sites is bounded away from zero uniformly in the size of the interval. The technique of proof of the second part is based on heuristic, ad hoc arguments and the first part.

For $(\lambda,\mu)$ such that $\mu>\mu_{c}$ and $\lambda >0$ it is also shown in \cite{DS} that the process survives. Assuming further that $\mu\geq\lambda$ and letting $r_{t} = \sup\{ x: \zeta_{t}^{O}(x)=1\}$,  it is shown in Tzioufas \cite{T} that $\displaystyle{ \frac{r_{t}}{t} \rightarrow \alpha}$, as $t \rightarrow \infty$, almost surely on $\{\zeta^{O}_{t} \textup{ survives}\}$ and also that $\alpha>0$. The constant $\alpha$ is referred to as \textit{the limit of the speed}. We prove the following comparison with the contact process result. 






\begin{thm}\label{thmsup}
Suppose that $\mu>\mu_{c}$ and $\mu \geq \lambda>0$. Let $\beta$ be the limit of the speed of the contact process with parameter $\mu$. Let also $\alpha$ be the limit of the speed of the three state contact process with parameters $(\lambda,\mu)$. We have that $\displaystyle{\alpha \leq (\lambda/ \mu) \beta}$. 
\end{thm}

Considering the process started from all sites on the negative half line infected and other sites susceptible and never infected, the proof of Theorem \ref{thmsup} is based on coupling with an a.s.\ infinite sequence of contact processes  appropriately defined on the trajectory of the process's rightmost infected site. We also note that the reason for the assumption on the parameters $\mu \geq \lambda$ is that the techniques of proof extensively use certain coupling results which hold in this case only (see e.g. Theorem 3.5). 




In the next section we explain the graphical construction and state some background results, while the remainder of the chapter is devoted to proofs. Theorem \ref{thmsub} is proved in Section \ref{Ssub} and Theorem \ref{thmsup} in Section \ref{Ssup}. 



\section{Preliminaries}\label{prel} 

Let $(\lambda,\mu)$ be fixed values of the parameters and suppose that $\mu \geq \lambda$, the other case is similar.  To carry out our construction for all sites $x$, $x \in \Z$, and $y=x-1,x+1$, let $\{T_{n}^{(x,y)}, n\geq1 \}$ and $\{U_{n}^{(x,y)}, n\geq1\}$ be the event times of Poisson processes respectively at rates $\lambda$ and $\mu-\lambda$; further, let $\{S_{n}^{x}, n\geq 1\}$ be the event times of a Poisson process at rate $1$. (All Poisson processes introduced are independent). 


The graphical construction will be used in order to visualize the construction of processes on the same probability space. Consider the space $\Z \times [0,\infty)$ thought of as giving a time line to each site of $\Z$; Cartesian product is denoted by $\times$. Given a realization of the before-mentioned ensemble of Poisson processes, we define the \textit{graphical construction} and $\zeta_{t}^{[\eta,s]}$,  $t\geq s$, the three state contact process with parameters $(\lambda,\mu)$ started from $\eta$ at time $s\geq0$, i.e.\ $\zeta_{s}^{[\eta,s]} = \eta$, as follows. From each point $x \times T_{n}^{(x,y)}$ we place a directed $\lambda$-\textit{arrow} to $y \times T_{n}^{(x,y)}$; this indicates that at all times $t=T_{n}^{(x,y)}$, $t\geq s$, if $\zeta_{t-}^{[\eta,s]}(x)=1$ and $\zeta_{t-}^{[\eta,s]}(y)=0$ \textit{or} $\zeta_{t-}^{[\eta,s]}(y) = -1$ then we set $\zeta_{t}^{[\eta,s]}(y) = 1$ (where $\zeta_{t-}(x)$ denotes the limit of $\zeta_{t-\epsilon}(x)$ as $\epsilon\rightarrow 0$). From each point $x \times U_{n}^{(x,y)}$ we place a directed $(\mu-\lambda)$-\textit{arrow} to $y \times U_{n}^{(x,y)}$; this indicates that at any time $t=U_{n}^{(x,y)}$, $t\geq s$, if $\zeta_{t-}^{[\eta,s]}(x)=1$ and $\zeta_{t-}^{[\eta,s]}(y)=0$ then we set $\zeta_{t}^{[\eta,s]}(y) = 1$. While at each point $x \times S_{n}^{x}$ we place a \textit{recovery mark}; this indicates that at any time $t= S_{n}^{x}, t\geq s,$ if $\zeta_{t-}^{[\eta,s]}(x)=1$ then we set $\zeta_{t}^{[\eta,s]}(x) = 0$. The special marks were introduced in order to make connection with percolation and hence the contact process. We say that a \textit{path exists} from $A \times s$ to $B \times t$,  $t\geq s$,  if there is a connected oriented path from $x \times s$ to $y \times t$, for some $x \in A$ and $y \in B$, that moves along arrows (of either type) in the direction of the arrow and along vertical segments of time-axes without passing through a recovery mark, we write $A \times s \rightarrow B \times t$ to denote this. Defining $\xi_{t}^{A\times s} := \{x: A \times s \rightarrow x \times t\}$, $t\geq s$, we have that $\xi_{t}^{A \times s}$ is the contact process with parameter $\mu$ started from $A$ at time $s$. To simplify our notation $\zeta_{t}^{[\eta,0]}$ will be denoted as $\zeta_{t}^{\eta}$; we also simply write $\xi_{t}^{A}$ for $\xi_{t}^{A \times 0}$. 





In the remainder of this section we collect together a miscellany of known results and properties that we will need to use, we briefly state them and give references for proofs and further information. An immediate consequence of the graphical construction we will use is \textit{monotonicity}: Whenever a certain path of the graphical representation exists from $A \times s$ to $B \times t$,  $t\geq s$, then for all $C\supseteq A$ the same path exists from $C\times s$ to $B \times t$. Another property of the contact process we use is self duality. If $(\xi_{t}^{A})$ and $(\xi_{t}^{B})$ are contact processes with the same infection parameter started from $A$ and $B$ respectively, then the following holds, for all $t\geq0$,
\begin{equation}\label{selfdual}
\pr(\xi_{t}^{A} \cap B \not= \emptyset) = \pr(\xi_{t}^{B} \cap A \not= \emptyset),
\end{equation}
the duality relation is easily seen by considering paths of the graphical construction that move along time axes in decreasing time direction and along infection arrows in direction opposite to that of the arrow and noting that the law of these paths is the same as that of the paths (going forward in time) defined above. 

Letting $I$ be an integer, we note that we simplify $\xi_{t}^{A} \cap \{I\} \not= \emptyset$ and write $\xi_{t}^{A} \cap I \not= \emptyset$ instead; we also write $I \times s$ for $\{I\} \times s$ .



The following is a well known exponential decay result for the subcritical contact process, see e.g.\ \cite{D88}.

\begin{thm}\label{thmpre1}
Let $\xi_{t}^{A}$ be the contact process with parameter $\mu$ started from $A$. 
If $\mu < \mu_{c}$ then there exists $\psi>0$, depending only on $\mu$, such that 
\[
\pr(\xi_{t}^{A} \not= \emptyset) \leq |A|  e^{- \psi t}, \mbox{ for all } t, 
\]
where $|A|$ denotes the cardinality of $A\subset \Z$.
\end{thm}

We also need the next result from Durrett \cite{D80}; see Lemma 4.1.

\begin{lem}\label{thmpresup}
Let $B$ be any infinite set such that $B \subseteq (-\infty,0]$. Consider the contact processes $\xi_{t}^{B}$ and $\xi_{t}^{B\cup\{1\}}$ with parameter $\mu$ coupled by the same graphical construction. Letting $R_{t}^{A} = \sup \xi_{t}^{A}$, we have that, for all $t$, $\displaystyle{ \E(R_{t}^{B\cup\{1\}} - R_{t}^{B}) \geq 1}$.

\end{lem}
%

Finally we quote two results concerning the three state contact process. The next result is in Section 5 of Stacey \cite{S}, see also Theorem \ref{PIPmon} below.

\begin{thm}\label{moninit}
Let $\eta$ and $\eta'$ be any two configurations such that $\eta(x) \leq \eta'(x)$ for all $x$. Consider  $\zeta_{t}^{\eta}$ and $\zeta_{t}^{\eta'}$ the corresponding  three state contact processes with parameters $(\lambda,\mu)$ coupled by the graphical construction. If $\mu \geq \lambda$, then $\zeta_{t}^{\eta}(x) \leq \zeta_{t}^{\eta'}(x)$ holds for all $x$ and $t$. We refer to this property as monotonicity in the initial configuration.
\end{thm}


The first part of the next statement is a special case of Theorem 4 of Durrett and Schinazi \cite{DS}; the second part is Theorem 1, part (i), in \cite{T}.

\begin{thm}\label{speedlim}
Let $\zeta_{t}^{O}$ be the three state contact process with parameters $(\lambda,\mu)$ started from the standard initial configuration, let also $r_{t} = \sup\{x: \zeta_{t}^{O}(x)=1\}$. If $(\lambda, \mu)$ are such that $\mu \geq \lambda>0$ and $\mu > \mu_{c}$, then the process survives and, a fortiori, there exists $\alpha>0$ such that $\displaystyle{\frac{r_{t}}{t} \rightarrow \alpha}$ almost surely on $\{\zeta_{t}^{O} \mbox{\textup{ survives}}\}$, we refer to $\alpha$ as the limit of the speed. 
\end{thm}

\section{Proof of Theorem \ref{thmsub}}\label{Ssub} 
In this section we establish Theorem \ref{thmsub} as the compound of two separate propositions. 
Recall that $\mu_{c}$ is the critical value of the contact process and that $|B|$ denotes the cardinality of $B\subset \Z$. 

\begin{lem}\label{Tparenthexp}
Let $\hat{\xi}_{t}^{A}$ be the contact process constrained on $\{\min A,\dots,\max A\}$ started from $A$. Consider  $\hat{\xi}_{t}^{A}$, $|A|<\infty$, with the same parameter $\mu$. For all $\mu<\mu_{c}$ there exists
$C,\gamma>0$ independent of $A$ such that 
\begin{equation*}\label{cpuni}
\pr\left( \exists \hspace{0.5mm} s\geq t \textup{ s.t., } \hat{\xi}_{s}^{A} \cap \min A \not= \emptyset \textup{ or } \hat{\xi}_{s}^{A} \cap \max A \not= \emptyset \right) \leq C e^{-\gamma t}, \hspace{2mm} \text{ for all } t\geq0.
\end{equation*}

\end{lem}

\begin{proof}
Consider the graphical construction on the integers for the contact process with parameter $\mu$ (i.e.\ there is only one type of arrows positioned according to event times of Poisson processes at rate $\mu$), and let $\mu< \mu_{c}$. 
Let $N\geq0$ be any finite integer and consider $\hat{\xi}_{t}^{[0,N]}$ defined by use of truncated paths containing vertical segments of time axes of sites within $[0,N]$ only. By monotonicity and translation invariance, it is sufficient to prove that there exist $C,\gamma>0$ independent of $N$ such that, 
\begin{equation}\label{lemcutpunc}
\pr\left( \exists \hspace{0.5mm} s\geq t \textup{ s.t., } \hat{\xi}_{s}^{[0,N]} \cap  0 \not= \emptyset \textup{ or } \hat{\xi}_{s}^{[0,N]} \cap N \not= \emptyset \right) \leq C e^{-\gamma t}, \hspace{2mm} \text{ for all } t\geq0.
\end{equation}

Define $E_{N,t} = \{ \hat{\xi}_{t}^{[0,N]} \cap N \not= \emptyset \textup{ or } \hat{\xi}_{t}^{[0,N]} \cap 0 \not= \emptyset \}$, $t \geq 0$. We first show that there exists a $\psi>0$ such that for any $N\geq0$,
\begin{equation}\label{subdual}
\pr(E_{N,t}) \leq 2 e^{-\psi t} \hspace{2mm} \mbox{ for all } t\geq0. 
\end{equation}
Recall that we denote by $\xi_{t}^{A}$ the contact process started from $A$. By duality equation (\ref{selfdual}) and translation invariance we have that there exists a $\psi>0$ such that for any $N\geq0$,
\begin{eqnarray*}\label{dualandmon}
\pr\big(\xi^{[0,N]}_{t} \cap N  \not= \emptyset\big) &=& \pr\big(\xi_{t}^{0} \cap [-N,0] \not= \emptyset\big)\\
&\leq&  \pr\big(\xi_{t}^{0} \cap \Z \not= \emptyset\big)\hspace{1mm} \leq \hspace{1mm} e^{-\psi t} 
\end{eqnarray*}
for all $t\geq0$, where the two inequalities come from monotonicity and Theorem \ref{thmpre1} respectively. The display above gives us (\ref{subdual}), by monotonicity and translation invariance. 

For all integers $k\geq1$, define the event $D_{N,k}$ to be such that $\omega \in D_{N,k}$ if and only if $\omega \in E_{N,s}$ for some $s \in (k-1 ,k]$. By the Markov property for $\hat{\xi}_{t}^{[0,N]}$, because the probability of no recovery mark on the time axis of $N$ or $0$ from the first time $s \in (k-1,k]$ such that $\omega \in E_{N,s}$ until time $k$ is at least $e^{-1}$, we have that for all $k\geq1$,
\begin{equation}\label{expminus}
e^{-1}\pr(D_{N,k}) \leq \pr(E_{N,k}).
\end{equation}

Considering the event $\bigcup\limits_{l\geq0}D_{N, l+\lfloor t \rfloor}$, where $\lfloor \cdot\rfloor$ denotes the floor function, Boole's inequality gives that 
\begin{equation*}
\pr\left( \exists \hspace{0.5mm} s\geq t \textup{ s.t., } \hat{\xi}_{s}^{[0,N]} \cap  0 \not= \emptyset \textup{ or } \hat{\xi}_{s}^{[0,N]} \cap N \not= \emptyset \right) \leq  \sum_{l\geq0} \pr(D_{N, l+\lfloor t \rfloor})
\end{equation*}
$t\geq0$,  by (\ref{expminus}) and then (\ref{subdual}), the last display implies (\ref{lemcutpunc}), thus the proof is completed. 
\end{proof}

We note that the preceding lemma is used into proving the next one as well as for proving Corollary \ref{Teta} below. 

\begin{lem}\label{ANScp}
Let $\tilde{\xi}_{t}^{A}$ be the contact process constrained on $\{\min A-1,\dots,\max A+1\}$ started from $A$. Consider $\tilde{\xi}_{t}^{A}, |A| < \infty$, with the same parameter $\mu$. For all $\mu<\mu_{c}$ there exists $\epsilon>0$ independent of $A$ such that $\pr\big( \forall \hspace{0.5mm} t\geq0, \tilde{\xi}_{t}^{A} \subseteq [\min A,\max A]\big) \geq \epsilon$.
\end{lem}

\begin{proof}
Consider the graphical construction for the contact process with parameter $\mu$, $\mu<\mu_{c}$. 
Let $N\geq0$ be any finite integer and consider $\tilde{\xi}_{t}^{[0,N]}$ defined by use of truncated paths containing vertical segments of time axes of sites only within $[-1,N+1]$. By monotonicity and translation invariance, it is sufficient to show there exists $\epsilon>0$ independent of $N$ such that
\begin{equation}\label{epsone}
\pr\big( \forall \hspace{0.5mm} t\geq0, \mbox{ }\tilde{\xi}^{[0,N]}_{t} \subseteq [0,N]\big) \geq \epsilon.
\end{equation}

Define $\tilde{E}_{N,t} = \{\tilde{\xi}_{t}^{[0,N]} \cap N+1 \not= \emptyset \textup{ or } \tilde{\xi}_{t}^{[0,N]} \cap -1 \not= \emptyset \}$, $t \geq 0$. We have that there is a $\psi>0$ such that for any $N\geq0$,
\begin{equation}\label{subdual2}
\pr(\tilde{E}_{N,t}) \leq 2 e^{-\psi t} \hspace{2mm} \mbox{ for all } t\geq0.
\end{equation}
The proof of (\ref{subdual2}) is derived by (\ref{subdual}) as follows. Let $\hat{\xi}_{t}^{[0,N]}$ and $E_{N,t}$ be as in Lemma \ref{Tparenthexp}, by monotonicity we have that $\tilde{\xi}_{t}^{[0,N]}$ is stochastically smaller than $\hat{\xi}_{t}^{[-1,N+1]}$. Hence $\pr(\tilde{E}_{N,t}) \leq \pr(E_{N+2,t})$ by translation invariance. Alternatively, one can in essence repeat the arguments used for the proof of (\ref{subdual}). 

Define $\tilde{D}_{N,k} = \{ \omega: \omega \in \tilde{E}_{N,s} \mbox{ for some } s \in (k-1,k]\}$, for integer $k\geq1$.  Note that $\textstyle \bigcap\limits_{k\geq1}\tilde{D}^{c}_{N,k}$ is equal to $\left\{ \forall \hspace{0.5mm} t\geq0, \tilde{\xi}^{[0,N]}_{t} \subseteq [0,N] \right\}$ and that $\pr(\textstyle \bigcap\limits_{k\geq1}\tilde{D}^{c}_{N,k}) = \textstyle \lim\limits_{K \rightarrow \infty} \pr( \bigcap\limits_{k\geq1}^{K}\tilde{D}^{c}_{N,k})$.  Since also $\{\tilde{D}^{c}_{N,k}\}_{k\geq1}^{K}$ are monotone decreasing and hence positively correlated, the Harris-FKG inequality (see e.g.\ \cite{D88}, \cite{L}) gives that
\[
\pr\big(\forall \hspace{0.5mm} t\geq0, \mbox{ }\tilde{\xi}^{[0,N]}_{t} \subseteq [0,N] \big)  \geq  \prod_{k \geq1} \pr(\tilde{D}^{c}_{N,k}).
\] 

However, by elementary properties of infinite products, $(\ref{subdual2})$ implies that there is $\epsilon>0$ independent of $N$ such that $\prod\limits_{k \geq 1} \big(1 - e\pr(\tilde{E}_{N,k})\big)>\epsilon$. Because, similarly to (\ref{expminus}), we have that $\pr(\tilde{D}_{N,k}) \leq e\pr(\tilde{E}_{N,k})$, which implies (\ref{epsone}) from the last display above, and thus completes the proof.  
\end{proof}



We return to consideration of the three state contact process. 

\begin{definition}\label{calI}
We denote by $\mathcal{I}(\zeta)$ the set of infected sites of some given configuration $\zeta$, that is,  $\mathcal{I}(\zeta) = \{y \in \Z:\zeta(y) =1\}$.
\end{definition}

To state the next result, for all $N\geq0$, let $\eta_{N}$ be such that $\mathcal{I}(\eta_{N}) = \{-N,\dots, N\},$ while all other sites in $\eta_{N}$ are susceptible and never infected. Note that $\eta_{0}$ is actually the standard initial configuration, and hence for $N=0$ the next result reduces to the first part of Theorem \ref{thmsub}.

\begin{prop}\label{geomb}
Consider $\zeta_{t}^{\eta_{N}}, N<\infty$, with parameters $(\lambda, \mu)$. For all $\mu<\mu_{c}$ and $\lambda<\infty$, there exists $\epsilon>0$ independent of $N$ such that 
\begin{equation*}\label{eqgeomb}
\pr\Big(\exists \hspace{0.5mm}t \textup{ s.t.}, \mbox{ }\zeta_{t}^{\eta_{N}}(N+n)=1  \textup{ or } \zeta_{t}^{\eta_{N}}(-N-n)=1 \Big) \leq (1-\epsilon)^{n}, \hspace{2mm} \mbox{ for all } n\geq1. 
\end{equation*}

\end{prop}

\begin{proof}
Consider the graphical construction for $(\lambda, \mu)$ as in the statement. Let $N\geq0$ be any finite integer and consider the process $\zeta_{t}^{\eta_{N}}$, let also $I_{t}^{N} =  \I(\zeta_{t}^{\eta_{N}})$. 
We first show that there exists $\epsilon>0$ independent of $N$ such that
\begin{equation}\label{targ}
\pr\big( \forall\hspace{0.5mm} t \geq 0, \mbox{ }I_{t}^{N} \subseteq  [-N, N] \big)\geq \epsilon. 
\end{equation}
Define the event $B_{N}= \left\{ \forall\hspace{0.5mm}s \in (0,1], I_{s}^{N}\subseteq [-N,N]\right\} \cap \left\{I_{1}^{N} \subseteq [-N+1,N-1] \right\}$ and also $F_{N} = \{\forall\hspace{0.5mm} t \geq 1, \mbox{ }I_{t}^{N} \subseteq  [-N+1, N-1]\}$. From Lemma \ref{ANScp} and the Markov property at time $1$, there exists $\epsilon>0$ independent of $N$ such that $\pr\big(F_{N}| B_{N} \big)\geq \epsilon$. From this and because also $\{\forall\hspace{0.5mm} t \geq 0, \mbox{ }I_{t}^{N} \subseteq  [-N, N] \}\supseteq F_{N} \cap B_{N}$, it is sufficient to show that $\pr(B_{N})$ is bounded away from zero uniformly in $N$. However we have that $	B_{N} \supseteq B_{N}'$, where $B_{N}'$ is the event that (a) no arrow exists from $N \times s$ to $N+1 \times s$ and from $-N\times s$ to $-N-1\times s$ for all times $s \in (0,1]$, (b) a recovery mark exists within $N \times (0,1]$ and $-N \times (0,1]$ and, (c) no arrow exists from $N-1 \times s$ to $N\times s$ and over $-N+1 \times s$ to $-N \times s$, for all times $s \in (0,1]$. (Note that (b) implies that there is a $t\in(0,1]$ such that $I_{t}^{N} \subseteq [-N+1,N-1]$ and (c) assures that this equation holds for $t=1$). Thus, because by translation invariance $\pr(B_{N}')$ does not depend in $N$ and is strictly positive, we get that (\ref{targ}) is proved. 

From (\ref{targ}) by monotonicity (of the contact process) we have that, indeed for any $\eta$ such that $\eta(x) \not=-1$, $\forall \hspace{0.1mm} x \in [\min\I(\eta), \max\I(\eta)]$, $\pr\big( \forall \hspace{0.5mm}t \geq 0,\mbox{ } \I(\zeta_{t}^{\eta}) \subseteq  \I(\eta) \big)\geq \epsilon$. Due to the nearest neighbours assumption, the proof is completed by $n$ repeated applications of  the Strong Markov Property and the last inequality.
 



\end{proof}

For the proof of Proposition \ref{propsub2} below we will use the previous proposition as well as the next corollary.
To state the latter the following definitions are needed. For any $\eta$ such that $|\I(\eta)|<\infty$, consider $\zeta_{t}^{\eta}$ and define the associated stopping time $T^{\eta}:= \inf\{t\geq0 : \I(\zeta_{t}^{\eta}) \not\subseteq [\min\I(\eta), \max\I(\eta)]\}$. Define also the collection of configurations $H = \{\eta:\eta(x) \not=-1, \mbox{ } \forall x \in [\min\I(\eta), \max\I(\eta)]\}$. We also note that the indicator of an event $E$ is denoted by $1_{E}$. 
\begin{cor}\label{Teta}
Consider $\zeta_{t}^{\eta}, \eta \in H$, with parameters $(\lambda, \mu)$. For all $\mu<\mu_{c}$ and $\lambda<\infty$, there exist $C$ and $\theta>0$ independent of $\eta \in H$ such that $\E(e^{\theta T^{\eta}1_{\{T^{\eta}<\infty\}}}) \leq C$.
\end{cor}

\begin{proof}
Let $\hat{\xi}_{t}^{A}$ be as in Lemma \ref{Tparenthexp}. 
For any $\eta \in H$ coupling $\zeta_{t}^{\eta}$ with parameters $(\lambda, \mu)$ with $\hat{\xi}_{t}^{A}$ with parameter $\mu$ and $A = \I(\eta)$ by the graphical representation gives that 
$\{t \leq T^{\eta} <\infty\} \subseteq \{\exists \hspace{0.5mm} s\geq t \textup{ s.t., } \hat{\xi}_{s}^{A} \cap \min A \not= \emptyset \textup{ or } \hat{\xi}_{s}^{A} \cap \max A \not= \emptyset\}$ holds, hence the proof follows from the before-mentioned lemma and the integral representation of expectation. 
\end{proof}

The next statement is the second part of Theorem \ref{thmsub}. Recall that $\zeta_{t}^{O}$ denotes the process started from the standard configuration. 

\begin{prop}\label{propsub2}
Consider $\zeta_{t}^{O}$ with parameters $(\lambda,\mu)$ and let $I_{t} = \I(\zeta_{t}^{O})$.
For all $\mu<\mu_{c}$ and $\lambda<\infty$, there exist $C$, $\gamma>0$ such that
$\pr(I_{t} \not= \emptyset) \leq C e^{-\gamma t}$ for all $t\geq0$. 
\end{prop}


\begin{proof} 
Consider the graphical construction for $(\lambda, \mu)$ as in the statement and let $S_{t} =  [\min I_{t} , \max I_{t} ] \cap \Z$, $t\geq0$, where by convention $\min \emptyset = \infty$. Define the stopping times $\tau_{k} = \inf\{t\geq0: |S_{t}| = k\}$, $k\geq1$, define also $K = \inf\{k: \tau_{k} = \infty\}$ and $\sigma_{K} = \inf\{s\geq 0 : I_{s+\tau_{K-1}} = \emptyset\}$. 
It is elementary that the sum of two exponentially bounded random variables is itself exponentially bounded, a simple proof can be obtained by the integral representation of expectation, Chernoff's bound and use of the Cauchy-Schwartz inequality. Because $\{I_{t} \not= \emptyset\}$ is the same as $\{\tau_{K-1}+\sigma_{K} \geq t\}$, it is sufficient to prove that the non-stopping time $\tau_{K-1}$, and $\sigma_{K}$ are both exponentially bounded.  Since $K$ is exponentially bounded by Proposition $\ref{geomb}$ and by set theory we have that, for all $a>0$,
\begin{equation*}\label{tauK}
\pr(\tau_{K-1} > t) \leq \pr(K > \lceil at \rceil ) + \pr(\tau_{K-1} > t, K \leq \lceil at \rceil ) 
\end{equation*}
$t\geq0$, it is sufficient to show that: (i) there is $a>0$ such that $\tau_{K-1}$ is exponentially bounded on $\{K \leq \lceil at \rceil\}$ and, by repeating the argument that led to the inequality of the last display above, (ii) $\sigma_{K}$ is exponentially bounded on $\{K \leq \lceil t \rceil\}$.


Let $H$ and $C,\theta>0$ be as in Corollary \ref{Teta}. By the Strong Markov Property because, due to the nearest neighbours assumption, $\zeta_{\tau_{k-1}}^{O} \in H$, we have that,  
\begin{eqnarray}\label{subindu}
\E( e^{\theta \tau_{k}1_{\{\tau_{k} < \infty\}}}) &\leq& \E(e^{\theta \tau_{k-1}1_{\{\tau_{k-1} < \infty\}}}e^{\theta (\tau_{k}-\tau_{k-1})1_{\{(\tau_{k}-\tau_{k-1}) < \infty\}}}) \nonumber \\
&\leq& C \E(e^{\theta \tau_{k-1}1_{\{\tau_{k-1} < \infty\}}})  \nonumber
\end{eqnarray}
$k\geq1$. Iterating the last inequality gives that $\E( e^{\theta \tau_{k}1_{\{\tau_{k} < \infty\}}}) \leq C^{k}$, thus, by set theory we have that, for all $a>0$,
\begin{eqnarray*}\label{tauN2}
\pr(\tau_{K-1} > t, K \leq \lceil at \rceil ) &\leq& \sum_{k=1}^{\lceil at \rceil}e^{-\theta t}\E( e^{\theta \tau_{k-1}1_{\{\tau_{k-1} < \infty\}}}) \nonumber\\
&\leq& \lceil at \rceil e^{-\theta t} C^{\lceil at \rceil}
\end{eqnarray*}
$t\geq0$, choosing $a>0$ such that $e^{-\theta}C^{\lceil a \rceil}<1$ we see that the right side of the last display is exponentially bounded in $t$, this proves (i). 

We prove (ii), let $(\hat{\xi}_{t}^{A})$ be the contact process at rate $\mu< \mu_{c}$ on the subsets of $\{\min A,\dots,\max A\}$ started from $A$. By coupling we have that $\{\sigma_{k}1_{\{K=k\}}\geq t\}$ is stochastically bounded above by $\{\hat{\xi}_{t}^{[1,k]} \not= \emptyset\}$, Theorem \ref{thmpre1} by monotonicity gives that $\sum_{k=1}^{\lceil t \rceil} \pr(\sigma_{k}>t, K=k )$ is exponentially bounded in $t$.
\end{proof}

\begin{remark}
\textup{ Note that Proposition \ref{geomb} implies that $\E|\zeta_{t}^{O}| \rightarrow 0$, as $t\rightarrow \infty$, by bounded dominated convergence. This combined with a subadditivity argument analogous to Proposition 1.1 in \cite{AJ}, or an adaptation of the method of proof in Theorem 6.1 \cite{GRI}, would imply Proposition \ref{propsub2}. However none of the approaches seems plausible due to the lack of a generic monotonicity property for the process when $\lambda >\mu$.}
\end{remark}

\section{Proof of Theorem \ref{thmsup}}\label{Ssup} 
Let $\zeta^{\bar{\eta}}_{t}$ be the three state contact process with parameters $(\lambda, \mu)$ started from $\bar{\eta}$ such that $\bar{\eta}(x) =1$ for all $x \leq 0$ and $\bar{\eta}(x) =-1$ otherwise, let also $\bar{r}_{t} = \sup\I(\zeta^{\bar{\eta}}_{t})$. Throughout this section we concentrate on the study of  $\zeta^{\bar{\eta}}_{t}$ and in particular on the study of its rightmost infected site $\bar{r}_{t}$,  the necessary association with $\zeta_{t}^{O}$ is provided by the next corollary. Recall that $\alpha= \alpha( \lambda, \mu)>0$ is the limit of the speed as in Theorem \ref{speedlim} and that $\mu_{c}$ is the contact process's critical value. 

\begin{cor}\label{asconv}
If $\mu\geq\lambda>0$ and $\mu>\mu_{c}$, then $\displaystyle{\frac{\bar{r}_{t}}{t} \rightarrow \alpha}$ almost surely.
\end{cor}

\begin{proof}
This result is deduced from Theorem \ref{speedlim} by use of the restart argument in Lemma 2.14 (Lemma 4.4 in \cite{T}) , since $Y_{N}$ and $T_{Y_{N}}$ there are almost surely finite from Proposition 2.12 (Proposition 4.2 in \cite{T})) of the same chapter (paper). 
\end{proof}

For proving Theorem \ref{thmsup} we will need Corollary \ref{corsupexp} below, which in turn requires the succeeding lemma. 

We note that the next proof goes through varying the ideas of the corresponding result for the right endpoint of the contact process (see e.g.\ Theorem 2.19 in \cite{L}) in order for the subadditive ergodic theorem to apply in this context.

\begin{lem}\label{subadd}
Let $\bar{x}_{t} = \sup_{s\leq t} \bar{r}_{s}$. If $\mu\geq\lambda$, then $\displaystyle{ \frac{\bar{x}_{n}}{n} \rightarrow a}$ almost surely, where  $\displaystyle{  a =  \inf_{n\geq 0} \frac {\E(\bar{x}_{n})} {n}}$,  $a \in [-\infty,\infty)$. If further $a> -\infty$, then $\displaystyle{ \frac{\bar{x}_{n}}{n} \rightarrow a}$ in $L^{1}$. 
\end{lem}

\begin{proof}


Consider the graphical construction for $(\lambda,\mu)$ such that $\mu\geq\lambda$.  For each integer $y$, let $\eta_{y}$ be such that $\eta_{y}(x) =1$ for all $x\leq y$, and $\eta_{y}(x) =-1$ for all $x\geq y+1$. Consider the process $\zeta^{\bar{\eta}}_{t}$ and let $s,u$ be such that $s \leq u$, consider further the coupled process $\zeta_{t}^{[\eta_{\bar{x}_{s}},s]}$ and define 
\[
\bar{x}_{s,u} = \max\{y: \zeta_{t}^{[\eta_{\bar{x}_{s}},s]}(y)=1 \mbox{ for some } t \in [s,u] \} - \bar{x}_{s},
\] 
note that $\bar{x}_{0,u} = \bar{x}_{u}$, since $\bar{x}_{0} = 0$. By monotonicity in the initial configuration, see Theorem \ref{moninit}, we have that $\zeta_{t}^{[\eta_{\bar{x}_{s}},s]}(x) \geq \zeta^{\bar{\eta}}_{t}(x)$ for all $t\geq s$ and $x$, hence, 
\[
\mbox{(a) } \hspace{5mm} \bar{x}_{0,s} + \bar{x}_{s,u} \geq \bar{x}_{0,u}.
\]
However, by translation invariance and independence of Poisson processes used in the construction, we have that  $\bar{x}_{s,u}$ is equal in distribution to $\bar{x}_{0,u-s}$ and is independent of $\bar{x}_{0,s}$. 
From this, we get that,  
\[\mbox{(b)} \hspace{5mm} \{\bar{x}_{(n-1)k,nk},n \geq 1\} \mbox{ are i.i.d.\ for all }  k\geq1, \]
and hence are stationary and ergodic, and that,
\[\mbox{(c)} \hspace{5mm} \{\bar{x}_{m,m+k} , k \geq 0\} =  \{\bar{x}_{m+1,m+k+1} , k \geq 0\} \mbox{ in distribution,  for all } m\geq1.\]
Finally, by ignoring recovery marks in the construction and standard Poisson processes results gives that 
\[
\mbox{(d)} \hspace{5mm} \E(\max\{\bar{x}_{0,1},0\} ) < \infty,
\] 
From (a)--(d) above we have that $\{\bar{x}_{m,n}, m \leq n\}$ satisfies the corresponding conditions of Theorem 2.6, VI., in \cite{L}, this completes the proof. 
\end{proof}


\begin{cor}\label{corsupexp}
If $\mu\geq\lambda>0$ and $\mu>\mu_{c}$, then $\displaystyle{\frac{\E\bar{r}_{t}}{t} \rightarrow \alpha}$. 
\end{cor}
\begin{proof}
Consider the graphical construction for $(\lambda,\mu)$ such that $\mu\geq\lambda$ and $\mu>\mu_{c}$. Lemma \ref{subadd} gives that there is an $a>0$ such that $\displaystyle{ \frac{\bar{x}_{n}}{n} \rightarrow a}$ in $L^{1}$, where $a>0$ since from Corollary \ref{asconv} we have that $\alpha>0$, and $a \geq \alpha$ because $\bar{r}_{n} \leq \bar{x}_{n}$. Thus, $\displaystyle{ \bar{x}_{n} / n}$ are uniformly integrable by the direct part of the theorem in section 13.7 of Williams \cite{W}. 

Hence also because $\bar{r}_{n} \leq \bar{x}_{n}$ we have that $\displaystyle{ \bar{r}_{n} / n}$ are uniformly integrable, which, combined with Corollary \ref{asconv}, gives that $\displaystyle{ \frac{\bar{r}_{n}}{n}\rightarrow \alpha}$ in $L^{1}$ by appealing to the reverse part of the before-mentioned theorem. This implies in particular that $\displaystyle{\frac{\E\bar{r}_{n}}{n} \rightarrow \alpha}$. From this, the extension along real times follows by noting that $\max\limits_{t\in (n,n+1]}(\bar{r}_{t} -\bar{r}_{n})$ and  $\max\limits_{t\in (n,n+1]}(\bar{r}_{n+1} -\bar{r}_{t})$ are both bounded above in distribution by $\Lambda_{\mu}[0,1)$, the number of arrivals in $[0,1)$ of a Poisson process at rate $\mu$. 
\end{proof}


\begin{proof}[proof of Theorem \ref{thmsup}]
Consider the graphical construction for $\mu\geq\lambda$ and $\mu>\mu_{c}$.  Let $\mathcal{F}_{t}$ denote the sigma algebra associated to the Poisson processes used in the construction up to time $t$. Consider the process $\zeta^{\bar{\eta}}_{t}$, we follow the trajectory of the rightmost infected site, ${r}_{t}$, and consider the times $t$ such that $\bar{r}_{t} = \bar{x}_{t}$ and a $\mu-\lambda$ arrow from $\bar{r}_{t}$ to $\bar{r}_{t}+1$ exists, at each of those times we consider the set of infected sites of $\zeta_{t}^{\bar{\eta}}$ and initiate a coupled contact process with parameter $\mu$ having this as starting set. 
Let $\upsilon_{0}=0$, $\xi_{t}^{0}= \xi_{t}^{\Zm}$, where $\Zm = \{0,-1,\dots\}$, and $R_{t}^{0} = R_{t} = \sup \xi_{t}^{\Zm}$; for all $n\geq1$, consider 
\begin{equation}\label{eq:upsilons}
\upsilon_{n} = \inf\{t \geq \upsilon_{n-1}: R_{t}^{n-1}= \bar{r}_{t} + 1\},
\end{equation}
and then let $R_{t}^{n} = \sup \xi_{t}^{n}$, for $\xi_{t}^{n}:= \xi_{t}^{\I(\zeta^{\bar{\eta}}_{\upsilon_{n}}) \times \upsilon_{n}}$, $t\geq \upsilon_{n}$. 
Note that 
\begin{equation}\label{Rtn}
 \bar{r}_{t} = R_{t}^{n-1}, \mbox{ for all } t \in [\upsilon_{n-1},\upsilon_{n}), 
\end{equation}
and also
\begin{equation}\label{Hupsilsup}
\xi_{\upsilon_{n}}^{n-1} = \xi_{\upsilon_{n}}^{n} \cup \{\bar{r}_{\upsilon_{n}}+1\}, \mbox{ for all } n\geq1.
\end{equation}

Let $F_{t} = \sup\{n: \upsilon_{n} \leq t\}$. To complete the proof of Theorem \ref{thmsup} 
it is sufficient to show that
\begin{equation}\label{fracpunch}
\E(F_{t})= \frac{\mu-\lambda}{\lambda} \E(\bar{x}_{t})
\end{equation}
and that 
\begin{equation}\label{RFt}
\E(R_{t}-\bar{r}_{t}) \geq \E F_{t}
\end{equation}
$t \geq 0$. To see this, note that (\ref{fracpunch}) implies that $\displaystyle{ \E(F_{t}) \geq \frac{\mu-\lambda}{\lambda} \E(\bar{r}_{t})}$, because $\bar{x}_{t} \geq \bar{r}_{t}$. This combined with (\ref{RFt}) gives that $\displaystyle{ \E \bar{r}_{t} \leq \frac{\lambda}{\mu}\mbox{ } \E R_{t}}$, which  implies the desired inequality by Corollary \ref{corsupexp}. 

 
%


We first prove (\ref{RFt}). Recall that $1_{E}$ denotes the indicator of event $E$. From $(\ref{Rtn})$ we have that, $R^{0}_{t}-\bar{r}_{t} = \sum_{n=1}^{\infty}(R^{n-1}_{t}-R_{t}^{n}) 1_{\{F_{t} \geq n\}}$. Thus, because by monotonicity of the contact process $R_{t}^{n-1} \geq R_{t}^{n}$, the monotone convergence theorem gives us that  
\[
\E(R_{t}-\bar{r}_{t}) = \sum_{n=1}^{\infty}\E\left( (R^{n-1}_{t}-R_{t}^{n}) 1_{\{F_{t} \geq n\}} \right),
\]
$t\geq0$. From the Strong Markov Property, because ${\{F_{t} \geq n\}} = \{\upsilon_{n} \leq t\} \in \mathcal{F}_{\upsilon_{n}}$,
and Lemma \ref{thmpresup}, which applies from $(\ref{Hupsilsup})$, we have that 
\[
\E\left( (R^{n-1}_{t}-R_{t}^{n}) 1_{\{F_{t} \geq n\}} \right) \geq \pr(F_{t} \geq n).
\]
Combining the two last displays above gives (\ref{RFt}).

For proving (\ref{fracpunch}) some extra work is necessary. Recall the setting of the graphical construction in Section \ref{prel}. Let $\tilde{T}_{1}:= T_{1}^{0,1}$, $\tilde{S}_{1}:= S_{1}^{0}$, $\tilde{U}_{1}:= U_{1}^{0,1}$ and also 
define the events $A_{1} = \{\min\{\tilde{T}_{1}, \tilde{S}_{1}, \tilde{U}_{1}\} = \tilde{U}_{1}\}$ and $B_{1}= \{\min\{\tilde{T}_{1}, \tilde{S}_{1}, \tilde{U}_{1}\}=\tilde{T}_{1}\}$. At time $\tau_{0}:= 0$ the first competition takes place in the sense that on $A_{1}$, $\bar{r}_{\tilde{U}_{1}}= 0$ and $R_{\tilde{U}_{1}}^{0}=1$, hence $\upsilon_{1} = \tilde{U}_{1}$; while on $B_{1}$, $\bar{r}_{\tilde{T}_{1}} = \bar{x}_{\tilde{T}_{1}}=1$. By inductively repeating this idea we have the following. For all $n\geq1$, consider
\[
\tau_{n} = \inf\{ t\geq \min\{\tilde{T}_{n}, \tilde{S}_{n}, \tilde{U}_{n}\}: \bar{r}_{t} = \bar{x}_{t}\}, 
\]
and let also $\tilde{T}_{n+1}=\inf\limits_{k\geq1}\{T_{k}^{(\bar{r}_{\tau_{n}},\bar{r}_{\tau_{n}}+1)}:T_{k}^{(\bar{r}_{\tau_{n}},\bar{r}_{\tau_{n}}+1)}> \tau_{n}\}$, the first time a $\lambda$-arrow exists from $\bar{r}_{\tau_{n}}$ to $\bar{r}_{\tau_{n}}+1$  after $\tau_{n}$, and $\tilde{U}_{n+1} = \inf\limits_{k\geq1}\{U_{k}^{(\bar{r}_{\tau_{n}},\bar{r}_{\tau_{n}}+1)}: U_{k}^{(\bar{r}_{\tau_{n}},\bar{r}_{\tau_{n}}+1)}> \tau_{n}\}$,  the first such time a $(\mu-\lambda)$-arrow exists, and further $\tilde{S}_{n+1} =\inf\limits_{k\geq1} \{S_{k}^{\bar{r}_{\tau_{n}}}: S_{k}^{\bar{r}_{\tau_{n}}} >  \tau_{n}\}$, the first time that a recovery mark exists on $\bar{r}_{\tau_{n}}$ after $\tau_{n}$. Define also the events $A_{n+1}:=  \{\tilde{U}_{n+1} < \min\{\tilde{T}_{n+1},\tilde{S}_{n+1}\}\}$ and $B_{n+1}:= \{ \tilde{T}_{n+1} < \min\{ \tilde{U}_{n+1}, \tilde{S}_{n+1}\}\}$. In the sense explained above, we analogously have that at time $\tau_{n}$ the $n+1$ competition takes place. 

Let $N_{t} = \sup\{n: \tau_{n} <t\}$, we have that $\bar{x}_{t} = \sum\limits_{n=1}^{N_{t}}1_{B_{n}}$ and, because $\upsilon_{n}$ can also be expressed as the first $\tilde{U}_{k}$ after $\upsilon_{n-1}$ such that $\tilde{U}_{k} < \min\{ \tilde{T}_{k}, \tilde{S}_{k}\}$, we also have that $F_{t} = \sum\limits_{n=1}^{N_{t}} 1_{A_{n}}$. However, by ignoring recovery marks, $R_{t}$ is bounded above (in distribution) by $\Lambda_{\mu}[0,t)$,  the number of arrivals over $[0,t)$ of a Poisson process at rate $\mu$, and $\bar{x}_{t}$ is bounded above by $\Lambda_{\lambda}[0,t)$, while also $D_{t}$, the total number of recovery marks appearing on the trajectory of the rightmost infected site by time $t$, is bounded above by $\Lambda_{1}[0,t)$. Hence, noting that $N_{t} \leq R_{t}+ \bar{x}_{t}+ D_{t}$ and elementary Poisson processes properties, we have that $\E(N_{t})<\infty$. From this, by the Strong Markov Property and emulating the proof of Wald's lemma, since conditional on $\zeta^{\bar{\eta}}_{\tau_{n}}$ the events $A_{n+1}$ and $B_{n+1}$ are independent of $\{N_{t} \geq n+1\}= \{N_{t} \leq n\}^{c} \in \mathcal{F}_{\tau_{n}}$, we have that $\displaystyle{\E(F_{t}) = \E(N_{t}) \frac{\mu-\lambda}{\mu+1}}$ and also that $\displaystyle{\E(\bar{x}_{t}) = \E(N_{t})\frac{\lambda}{\mu+1}}$ from elementary results for competing Poisson processes. The last two equalities imply (\ref{fracpunch}), hence completing the proof. 
\end{proof}

\chapter{A note on Mountford and Sweet's extension of Kuczek's argument to non-nearest neighbours contact processes}

\paragraph{Abstract:} An elementary proof of the i.i.d.\ nature of the growth of the right endpoint for contact processes on the integers with symmetric, translation invariant interaction is presented. A related large deviations result for the density of oriented percolation is also given. 
\vfill

\section{Introduction}


The central limit theorem for the right endpoint of the contact process on the integers with nearest neighbours interaction (in other words, the basic one-dimensional contact process) was established in Galves and Presutti \cite{GP}. An alternative proof was later given in Kuczek \cite{K}. The seminal argument invented there is the embedding of regenerative space-time points, termed break points, on the trajectory of the right endpoint. By adapting Kuczek's argument and creation of a block construction that uses ideas from Bezuidenhout and Grimmett \cite{BG}, the result was extended to one-dimensional non-nearest neighbours contact processes in Mountford and Sweet \cite{MS}. The key to the extension, Theorem 3 in \cite{MS}, is that with positive probability the right endpoint of the process started from the origin is not overtaken from the right endpoint of the process started from all sites left of the origin for all times. 
%



In Section \ref{Scp} we aim at giving a short and complete proof of this theorem for contact processes on the integers with symmetric, translation invariant interaction; this result is then shown to be sufficient for giving an elementary proof of the i.i.d.\ behaviour of the right endpoint by a simple restart argument and the adaptation of Kuczek's argument in \cite{MS}. As a byproduct of an intermediate step for the former proof we see that whenever the shape theorem for the contact process holds, there is a positive probability that the process started from all sites and the process started from any finite set to agree on this set for all times. It is also worth stressing that the approach for showing the i.i.d.\ behaviour result does not require any renormalization arguments other than those in Durrett and Schonmann \cite{DS} used for the proof of the shape theorem, and further that, in order to extend the result to the central limit theorem the exponential estimate concerning the time of occurrence of a break point (Lemma 6 in \cite{MS}) is necessary. 



%
%



In Section \ref{Sperc}, we observe that a simple consequence of the result of Durrett and Schonmann \cite{DS2} for oriented percolation is a sharpened large deviations result than the one that the block construction in \cite{MS} builds upon, and remark on that the corresponding large deviations result for contact processes can be obtained in a simple manner.

\section{Contact processes}\label{Scp} 
The \textit{contact process} on a graph $G=(V,E)$ is a continuous time Markov process  $\xi_{t}$ whose state space is the set of subsets of $V$. Regarding each site in $\xi_{t}$ as occupied by a particle and all other sites as vacant, the process at rate $\mu$ evolves according to the following local prescription: (i) Particles die at rate 1. (ii) A particle at site $x$ gives birth to new ones at each site $y$ such that $xy \in E$ at rate $\mu$. (iii) There is at most one particle per site, that is, particles being born at a site that is occupied coalesce for all subsequent times.
Thus $\xi_{t}$ can be thought of as the particles descending from the sites in $\xi_{0}$. The contact process was first introduced in Harris \cite{H74} and has been widely studied since then; an up-to-date account of main results and proofs can be found in Liggett \cite{L99}. Let us denote by $\mu_{c}(G)$ the critical value of the contact process on $G$, that is $\mu_{c}(G) = \inf\{\mu: \pr(\xi_{t} \not=\emptyset, \mbox{ for all } t)>0\}$, where $\xi_{t}$ is the contact process on $G$ started from any $\xi_{0}$ finite, $\xi_{0} \subset V$. We note that throughout the proofs of this section we make extensive use of the construction of contact processes from the graphical representation, the reader is then assumed to be familiar with that and standard corresponding terminology (see \cite{D91} or \cite{L99}). 

We will consider the collection of simple graphs $Z_{M}$, $M\geq1$, where $M$ is a finite integer and $Z_{M}$ is the graph with set of vertices the integers, $\Z$, for which pairs of sites at Euclidean distance not greater than $M$ are connected by an edge. We shall also consider the related  collection of graphs $\Zmg_{M}$, $M\geq1$, where $\Zmg_{M}$ is the graph with set of vertices $\Zm :=\{0,-1,\dots\}$ obtained by $Z_{M}$ by retaining only edges connecting sites in $\Zm$. 






Firstly, the shape theorem for contact processes on $\Zmg_{M}, M\geq1,$ is stated, the result is a consequence of Durrett and Schonmann \cite{DS1}. Let us denote by $1(\cdot)$ the indicator function. 

\begin{thm}\label{SHAPE1}
Let $\hat{\xi}^{\Zm}_{t}$ and $\hat{\xi}_{t}^{F}$ denote the contact processes on $\Zmg_{M}, M\geq1$, at rate $\mu$ started from $\Zm$ and $F$ respectively. For all $M$, 
if $\mu>\mu_{c}(\Zmg_{M})$ and $F$ is finite then there is an $a>0$ such that 
the set of sites $x$ such that $1(x \in \hat{\xi}_{t}^{F})  = 1(x \in \hat{\xi}^{\Zm}_{t})$ contains $[-a t,0] \cap \Zm$ eventually, almost surely on $\{\hat{\xi}_{t}^{F} \not= \emptyset, \mbox{ for all } t\}$.
\end{thm}

\begin{proof}
It suffices to consider the case that $F=\{0\}$, the arguments given will be seen to apply for any $F=\{x\}$. Then, extension to all finite sets $F$ is immediate by additivity. To simplify the notation let us write $\hat{\xi}_{t}(x)$ for $1(x \in \hat{\xi}_{t})$. From the renormalized bond construction and the arguments of section 6 in \cite{DS1} we have that there are $C, \gamma \in (0,\infty)$ so that, for any $x \geq -at$,
\begin{equation}\label{fds}
\pr(\hat{\xi}_{t}^{0}(x) \not= \hat{\xi}^{\Zm}_{t}(x), \hat{\xi}_{t}^{0} \not= \emptyset) \leq Ce^{-\gamma t}
\end{equation}
$t\geq0$. Note that for integer times the result follows from (\ref{fds}) and the 1st Borel-Cantelli lemma since $\sum\limits_{n\geq1} \pr\big(\bigcup_{x\geq -an} \hat{\xi}_{n}^{0}(x) \not= \hat{\xi}^{\Zm}_{n}(x) | \hspace{1mm}  \hat{\xi}_{t}^{0} \not= \emptyset, \mbox{ for all } t\big) < \infty$, where we first used that  $\pr(\hat{\xi}_{t}^{0} \not= \emptyset \cap  \{\hat{\xi}_{t}^{0} \not= \emptyset, \mbox{ for all } t\}^{c})$ is exponentially bounded in $t$, the last standard result is proved by a standard argument, see e.g.\ the proof of Theorem 2.30 (a) in \cite{L99}. 

To obtain the result for real times, for any site $x$ let $B_t^{x}$ denote the event that $\bigcup_{t\in (n,n+1]} \{\hat{\xi}_{t}^{0}(x) \not= \hat{\xi}^{\Zm}_{t}(x) , \hat{\xi}_{t}^{0} \not= \emptyset\}$, and note that, 
\begin{equation}\label{BteM}
\textstyle{ \pr(B_t^{x}) e^{-2M\mu - 2} \leq  \pr(\hat{\xi}_{n+1}^{0}(x) \not= \hat{\xi}^{\Zm}_{n+1}(x),  \hat{\xi}_{n+1}^{0} \not= \emptyset)}
\end{equation}
where this inequality follows from the Strong Markov Property by letting $t_{0}$ be the first time such that $B_{t}^{x}$ occurs and considering the event that: (i) no particles attempt to occupy $x$ during $[t_{0},t_{0}+1]$, (ii) the particle of $\hat{\xi}^{\Zm}_{t_{0}}$ at $x$ does not die until $t_{0}+1$ and, (iii) one particle of $\hat{\xi}_{t}^{0}$ does not die until $t_{0}+1$. From (\ref{BteM}) combined with (\ref{fds}), the result follows as in the discrete time case by simply noting that the event $\left\{ \exists \hspace{0.3mm} t_{m} \uparrow \infty:\bigcup_{x\geq -at_{m}} \hat{\xi}_{t_{m}}^{0}(x) \not= \hat{\xi}^{\Zm}_{t_{m}}(x)\right\}$ can also be written as $\left\{\exists \hspace{0.3mm} n_{k}  \uparrow \infty: \bigcup_{t\in (n_{k},n_{k}+1]} \bigcup_{x\geq -at} \{\hat{\xi}_{t}^{0}(x) \not= \hat{\xi}^{\Zm}_{t}(x)\}\right\}$. 
\end{proof}


The foregoing shape theorem plays a pivotal role in establishing the next result that will be central in the proof of the main theorem of this section, viz. Theorem \ref{Rr}. We believe this to be of independent interest (see also Remark \ref{remark1}).

%
\begin{prop}\label{nnnkuc1} 
Let $\hat{\xi}^{\Zm}_{t}$ and $\hat{\xi}_{t}^{F}$ denote the contact processes  on  $\Zmg_{M}, M\geq1$, at rate $\mu$ started from $\Zm$ and $F$ respectively. For all $M$, if $\mu>\mu_{c}(\Zmg_{M})$ and $F$ is finite then  $\big\{\hat{\xi}_{t}^{\Zm} \cap F = \hat{\xi}_{t}^{F} \cap F , \mbox{\textup{ for all }}t \big\}$ has positive probability.
\end{prop}
\begin{proof} Fix $M$ and $F$ finite. Let $\mu>\mu_{c}(\Zmg_{M})$ and consider the processes $\hat{\xi}^{\Zm}_{t}$ and $\hat{\xi}_{t}^{F}$ constructed by the same graphical representation. Let $B_{n}$ denote the event $\{ \hat{\xi}^{\Zm}_{s} \cap F =\hat{\xi}^{F}_{s} \cap F, \mbox{ for all } s \geq n \}$, for all integers $n\geq0$.  

We give some notation. A realization of the graphical representation is typically denoted by $\omega$ and, we write that for all $\omega \in E_{1}$, $\omega \in E_{2}$ a.e.\ for denoting that $\pr(\{\omega: \omega \in E_{1}, \omega \not\in E_{2}\}) = 0$, where a.e.\ is an abbreviation for "almost everywhere" (on $E_{1}$). 

Theorem $\ref{SHAPE1}$ states that for all $\omega \in \{\hat{\xi}_{t}^{F} \not= \emptyset, \mbox{ for all } t\}$ there is an $s_{0}$ such that $\omega \in \{\hat{\xi}^{\Zm}_{s} \cap [-a s,0] = \hat{\xi}^{F}_{s} \cap [-a s,0] \mbox{, for all }s\geq s_{0}\}$ a.e.. Thus also, since $[-a s,0] \supset F$ for all $s$ sufficiently large,  for all $\omega \in \{\hat{\xi}_{t}^{F} \not= \emptyset, \mbox{ for all } t\}$ there is an $s_{1}$ such that $\omega \in  B_{\lceil s_{1} \rceil}$ a.e., where $\lceil s_{1} \rceil$ denotes the smallest integer greater than $s_{1}$.  Hence $\displaystyle{\pr\left(\cup_{n\geq 0} B_{n}\right) = \pr(\hat{\xi}_{t}^{F} \not= \emptyset, \mbox{ for all } t)>0}$, where the right side is strictly positive because $\mu>\mu_{c}(\Zmg_{M})$. From this we have  (e.g.\ by contradiction) that there is $n_{0}$ for which $\pr( B_{n_{0}}) >0$. We show that the last conclusion implies that $\pr(B_{0})>0$, this completes the proof.

Let $B_{n_{0}}'$ denote the event such that $\omega'\in B_{n_{0}}'$ if and only if there exists $\omega\in B_{n_{0}}$ such that  $\omega$ and $\omega'$ are identical realizations except perhaps from any $\delta$-symbols (death events) in $F \times (0,n_{0}]$. Further, let $D$ denote the event that no $\delta$-symbols exist in $F \times (0,n_{0}]$. 
By independence of the Poisson processes in the graphical representation and then because $B_{n_{0}}' \supseteq B_{n_{0}}$, we have that
\begin{eqnarray*}
\pr(B_{n_{0}}' \cap D) &=& \pr(B_{n_{0}}') \pr(D) \nonumber \\
&\geq& \pr(B_{n_{0}}) e^{-|F| n_{0}} >0,  
\end{eqnarray*}
where $|F|$ denotes the cardinality of $F$, because $B_{0} \supseteq B_{n_{0}}' \cap D$ the proof is completed from the last display. To prove that $B_{0} \supseteq B_{n_{0}}' \cap D$, note that if $\omega$ and $\omega'$ are identical except that $\omega'$ does not contain any $\delta$-symbols that possibly exist for $\omega$ on $F \times (0,n_{0}],$ then $\omega \in B_{n_{0}}$ implies that $\omega'\in B_{n_{0}}$ and indeed $\omega'\in B_{0}$. 
\end{proof}


\begin{remark}\label{remark1}
\textup{The arguments of the preceding proof readily apply in order to obtain the analogue of Proposition \ref{nnnkuc1} for the contact process on graphs such that the shape theorem holds. (Most prominent example is $\Z^{d}$, see \cite{BG}, \cite{D91}). Further, an explicit proof of the concluding sentence in the proof of Theorem 3 in \cite{MS} can be obtained by an argument along the lines of that given in the final paragraph of the preceding proof.}
\end{remark}









The next statement is the other ingredient we shall need in our proof. It is a consequence of the construction of Durrett and Schonmann \cite{DS1} comparison result, we also note that the result first appeared in Durrett and Griffeath \cite{DG} for the nearest neighbours case (see (b) in Section 2) .


\begin{thm}\label{muM}
For all $M$, $\mu_{c}(Z_{M})= \mu_{c}(\Zmg_{M})$.  
\end{thm}  




 We are now ready to state and prove the main result of this section. 


\begin{thm}\label{Rr} 
Let $\xi_{t}^{0}$ and $\xi^{\Zm}_{t}$ denote the contact processes  on $Z_{M},M\geq1,$ at rate $\mu$ started from $\{0\}$ and $\Zm$ respectively; let also $r_{t} = \sup \xi_{t}^{0}$ and $R_{t}=\sup \xi^{\Zm}_{t}$. For all $M$, if $\mu> \mu_{c}(Z_{M})$ then $\big\{ r_{t} = R_{t}, \textup{ for all } t \big\}$ has positive probability.
\end{thm}
\begin{proof} 
Fix $M$ and let $\mu>\mu_{c}(Z_{M})$. Let $\xi_{t}^{\M}$ and $\xi_{t}^{\Zm \backslash \M}$ be the contact process on $Z_{M}$ at rate $\mu$ started from $\M$, $\M := \{0,-1,\dots, -M-1\}$, and $\Zm \backslash \M$ respectively, further let $r_{t}^{\M} =\sup \xi_{t}^{\M}$. We will first show that
\begin{equation}\label{rtMcal}
\pr(r_{t}^{\M} = R_{t} , \mbox{ for all } t)>0.
\end{equation}
Consider  $\xi^{\Zm}_{t}$, $\xi_{t}^{\M}$ and $\xi_{t}^{\Zm \backslash \M}$ constructed by the same graphical representation and define the event 
\[
C=\{\xi_{t}^{\M} \cap \M  \supseteq \xi_{t}^{\Zm \backslash \M} \cap \M, \mbox{\textup{ for all }} t \}.
\]
Furthermore, let $\hat{\xi}_{t}^{\M}$ and $\hat{\xi}_{t}^{\Zm \backslash \M}$ be the contact process on $\Zmg_{M}$ at rate $\mu$ started from $\M$ and $\Zm \backslash \M$ respectively constructed by the same graphical representation as well (this is done by neglecting arrows from $x$ to $y$ such that $x \in \M$ and $y \in \{1,2\dots\}$ for all times). Let also $C'= \{\hat{\xi}_{t}^{\M} \cap \M  \supseteq \hat{\xi}_{t}^{\Zm \backslash \M} \cap \M, \mbox{\textup{ for all }} t \}$, by coupling and then monotonicity we have that 
\begin{equation}\label{Csupsets}
C= \{\xi_{t}^{\M} \cap \M  \supseteq \hat{\xi}_{t}^{\Zm \backslash \M} \cap \M, \mbox{\textup{ for all }} t \} \supseteq C'.
\end{equation}

Let $Z_{M}^{1^{+}}$ be the graph with sites $\Z^{1^{+}}:=\{1,2,\dots\}$, obtained by $Z_{M}$ by retaining all edges among sites in $\Z^{1^{+}}$. Let $\tilde{\xi}_{t+1}^{1\times 1}, t\geq0,$ be the contact process on $Z_{M}^{1^{+}}$ started from $\{1\}$ at time $1$ again constructed by the same graphical representation. Let $S$ be the event that there exists an arrow from some point in $\M \times [0,1]$ to some point in $\{1\} \times [0,1]$ intersected with $\{\tilde{\xi}_{t+1}^{1\times 1} \not= \emptyset, \mbox{ for all } t\geq0\}$. Let also $D$ denote the event that no $\delta$-symbols exist in $\M \times [0,1]$. Note that on $S \cap D$ we have that $\{r_{t}^{\M} \geq 0, \mbox{ for all } t\geq0\}$. 
From this and additivity we have that
\begin{equation}\label{Cr}
\{r_{t}^{\M} = R_{t} , \mbox{ for all } t\} \supseteq  C \cap S \cap D.
\end{equation}
From $(\ref{Csupsets})$ and the last display above, it is sufficient to show that $\pr(C'\cap S \cap D)>0$. 
However by independence of the Poisson processes in the graphical representation, we have that the events $C' \cap D$ and $S$ are independent. Further, from Theorem \ref{muM} and Proposition \ref{nnnkuc1} applied for $F=\M$, monotonicity and the Markov Property give that $\pr(C' \cap D)>0$. In addition, by Theorem \ref{muM} and translation invariance, using the Markov Property implies that $\pr(S)>0$. Thus (\ref{rtMcal}) is proved.  

To complete the proof, let $\xi_{t}^{0}$ and $\xi^{\Zm}_{t}$ be constructed by the same graphical representation,  considering $\{\xi_{1}^{0} \supseteq \M\} \cap \{  \xi_{s}^{0} \cap \{0\} \not= \emptyset \mbox { and } R_{s} \leq 0, \mbox{ for all } s \in (0,1] \}$, the result follows from monotonicity and the Markov property.
\end{proof}

 

The final result of this section addresses the i.i.d.\ nature of the growth of the right endpoint, which is the corresponding extension of the first part of the Theorem in Kuczek \cite{K}. 

 
\begin{thm}\label{THEthm}
Let $\xi_{t}^{0}$ denote the contact processes on $Z_{M}, M\geq1,$ at rate $\mu$ started from $\{0\}$, let also $r_{t} = \sup \xi_{t}^{0}$. For all $M$, if $\mu > \mu_{c}(Z_{M})$ then on $\{\xi_{t}^{0} \not= \emptyset, \mbox{ for all } t\}$ there are strictly increasing random  (but not stopping) times $\psi_{k}, k\geq0,$ such that $(r_{\psi_{n}} - r_{\psi_{n-1}}, \psi_{n} - \psi_{n-1})_{n\geq 1}$ are i.i.d.. 
\end{thm}

\begin{proof}
Fix $M$ and let $\mu>\mu_{c}(Z_{M})$. Consider the graphical representation for contact processes at rate $\mu$ on $Z_{M}$. Given a space-time point $x \times s$, let $\bar{\xi}^{x\times s}_{t+s}, t\geq 0,$ denote the process started from $\{y: y\leq x\}$ at time $s$ and let also $R_{t+s}^{x \times s}= \sup\bar{\xi}^{x\times s}_{t+s}$; furthermore let $\xi^{x\times s}_{t+s}, t\geq 0$, denote the process started from $\{x\}$ at time $s$, and let also $r_{t+s}^{x \times s} = \sup\xi^{x\times s}_{t+s}$. We write that $x \times s \mbox{ c.s.e.}$ for $R_{u}^{x \times s}= r_{u}^{x \times s}, \mbox{ for all } u\geq 0$, where the shorthand c.s.e.\ stands for "controls subsequent edges". 

By Theorem $\ref{Rr}$ we have that $p:= \pr( 0 \times 0\mbox{ c.s.e.})>0$. From this and the next lemma the proof follows by letting $\psi_{n} = \inf\{t\geq 1+\psi_{n-1}: r_{t} \times t \mbox{ c.s.e.}\},$ $n\geq0$, $\psi_{-1}:=0$, by elementary, known arguments, as in Lemma 7 in \cite{MS}.

\begin{lem}\label{restrt}
Consider the non stopping time $\psi = \inf\{t\geq 1: r_{t} \times t \mbox{ \textup{c.s.e.}}\}$. 
We have that $\psi$ and $r_{\psi}$ are a.s.\ finite conditional on either $\{\xi_{t}^{0}\not= \emptyset, \textup{ for all } t\}$ or $\{0 \times 0 \textup{ c.s.e.}\}$. 
\end{lem} 
\end{proof}

\begin{proof}[Proof of Lemma \ref{restrt}]
We define the sequence of processes $\xi_{t}^{n}, n\geq0,$ as follows. Consider $\xi_{t}^{0}:= \xi_{t}^{0 \times 0}$ and let $T_{0} = \inf\{ t: \xi_{t}^{0} = \emptyset\}$; inductively for all $n\geq0$, on $T_{n}<\infty$, let $\xi_{t}^{n+1}:= \xi_{t}^{0 \times T_{n}}, t\geq T_{n},$ and take $T_{n+1} = \inf\{t\geq T_{n}: \xi_{t}^{n+1} = \emptyset\}$. 

Let $r_{t}^{n} = \sup\xi_{t}^{n}$ and consider $r_{t}' := r_{t}^{n}$ for all $t\in [T_{n-1}, T_{n})$, where $T_{-1}:=0$. Let $\tau_{1}=1$ and inductively for all $n\geq1$, on $\tau_{n}<\infty$, let $\sigma_{n} := \sum_{k=1}^{n} \tau_{k}$ and $\tau_{n+1} = \inf\{t \geq 0: R_{t}^{r'_{\sigma_{n}} \times \sigma_{n}} > r_{t}^{r'_{\sigma_{n}} \times \sigma_{n}}\}$, while on $\tau_{n} = \infty$ let $\tau_{l} = \infty$ for all $l\geq n$. Let also $N = \inf\{n\geq1: \tau_{n+1} = \infty\}$. Since on $\{\xi_{t}^{0} \not= \emptyset, \mbox{ for all } t\}$, and on its subset $\{0 \times 0 \mbox{ c.s.e.}\}$, we have that $\psi = \sigma_{N}$ and $r'_{\sigma_{N}} = r_{\psi}$, it is sufficient to prove that $\sigma_{N}, r'_{\sigma_{N}}$ are a.s.\ finite.  

We prove the last claim. Note that, by translation invariance and independence of Poisson processes in disjoint parts of the graphical representation, we have that for all $n\geq1$ the event $\{\tau_{n+1}= \infty\}$ has probability $p$ and is independent of the graphical representation up to time $\sigma_{n}$. This and Bayes's sequential formula give that $\pr(N =n) = p(1-p)^{n-1}$ and, in particular, $N$ is a.s.\ finite. Thus also $\sigma_{N}$ is a.s.\ finite, which implies that $r'_{\sigma_{N}}$ is a.s.\ finite because $|r'_{t}|$ is bounded above in distribution by the number of events by time $t$ of a Poisson process at rate $M\mu$. This completes the proof. 


\end{proof}


\section{Large deviations}\label{Sperc}

We consider 1-dependent oriented site percolation with density at least $1-\epsilon$, that is, letting 
$\L =\{ (y,n) \in \Z^{2}:  y+n \mbox{ is even}, n\geq0\}$, a collection of random variables $w(y,n) \in \{0,1\}$ such that $(y,n) \in \L$ and $n\geq1$, which satisfies the property that $\pr \big( w(y_{i},n+1) =0 \mbox{ for all } 1\leq i \leq I|\mbox{} \{w(y,m), \mbox{ for all } m\leq n\} \big) \leq \epsilon^{I}$, where $|y_{i} - y_{i'}| > 2$ for all $1\leq i \leq I$ and $1\leq i' \leq I$ . Given a realization of 1-dependent site percolation we write $(x,0) \rightarrow(y,n)$, if there exists $x:= y_{0},\dots,y_{n}:=y$ such that $|y_{i} - y_{i-1}|=1$ and $w(y_{i}, i)= 1$ for all $1\leq i \leq n$. Let $2\Z = \{x: (x,0)\in \L\}$, for any given $A \subseteq 2\Z$, consider $W_{n}^{A} = \{y:(x,0)  \rightarrow (y,n) \mbox{ for some } x\in A\}$. Let also $2\Z+1 = \{x: (x,1)\in \L\}$, and define $X(n)$ to be $X(n)= 2\Z$ for even $n$, while $X(n)= 2\Z+1$ for odd $n$. Subsequently $C$ and $\gamma$ will represent positive, finite constants. 


The next lemma is used in the proof of the main result of this section below. It is a consequence of the result of Durrett and Schonmann \cite{DS2}. 

\begin{lem}\label{ldlem}
For all $\rho<1$ there is $\epsilon>0$ such that for any $n\geq1$ and $Y$, $Y \subset X(n)$, the probability of $\Big\{\sum\limits_{y \in Y} 1(y \in W^{2\Z}_{n}) < \rho |Y| \Big\}$ is bounded by $C e^{-\gamma |Y|}$.
\end{lem}

\begin{proof}  

We first consider standard independent bond percolation process, $B_{n}$, where $B_{n} \subset X(n)$, and let $p_{c}$ denote its critical value, for definitions see \cite{L99,D84}, the next lemma is proved immediately afterwards.

\begin{lem}\label{cor2z}
Let $B_{n}^{2\Z}$ be independent bond percolation process with parameter $p>p_{c}$ started from $2\Z$. For all $p'<p$ and any $n\geq1$ and $Y$, $Y \subset X(n)$, the probability of $\Big\{\sum\limits_{y \in Y} 1(y \in B_{n}^{2\Z}) < p' |Y| \Big\}$ is bounded by $Ce^{-\gamma |Y|}.$
\end{lem}

The proof then follows because we can choose $\epsilon>0$ sufficiently small such that $W^{2\Z}_{n}$ stochastically dominates $B_{n}^{2\Z}$ with parameter $p$ arbitrarily close to 1, which comes by combining Theorem B24 and Theorem B26 in \cite{L99}. 
\end{proof}

\begin{proof}[Proof of Lemma \ref{cor2z}]
Let $p>p_{c}$, let also $\tilde{B}_{n}$ be independent bond percolation process with parameter $p$ started from $\tilde{B}_{0}$ which is distributed according to the upper invariant measure of the process. By monotonicity we easily 
have $B^{2\Z}_{n}$ stochastically dominates $\tilde{B}_{n}$. From this, the proof follows by the invariance of $(\tilde{B}_{n})$ and the analogue of Theorem 1 in \cite{DS2} in this case.
\end{proof}

We now state and prove the main result of the section.
\begin{prop}\label{1depmain}
For all $\rho<1$ and all $\beta<1$ there is $\epsilon>0$ such that for any $n\geq1$ and $Y$, $Y \subset X(n) \cap [-\beta n,\beta n]$, the probability of $\Big\{\sum\limits_{y \in Y}1(y \in W_{n}^{0})<\rho |Y|, W_{n}^{0}\not= \emptyset\Big\}$ is bounded by  $Ce^{-\gamma n} + Ce^{-\gamma |Y|}$.
\end{prop}

\begin{proof}
Let $\tau = \inf\{n: W_{n}^{0} = \emptyset\}$, let also $R_{n} = \sup W_{n}^{0}$ and $L_{n} = \inf W_{n}^{0}$.
The following sequence of lemmas are known results, we refer to \cite{D84} and \cite{L99} for proofs. 

\begin{lem}\label{coupWW}
On $\{\tau = \infty\}$, $W_{n}^{0} = W^{2\Z}_{n} \cap [L_{n},R_{n}]$.
\end{lem}
\begin{lem}\label{taun}
There is $\epsilon>0$ such that for any $n\geq1$ the probability of $\{n \leq \tau < \infty\}$ is bounded by $Ce^{-\gamma n}$.
\end{lem} 
\begin{lem}\label{En} 
For all $\beta<1$ there is $\epsilon>0$ such that for any $n\geq1$ the probability of $\{[L_{n}, R_{n}] \subseteq [-\beta n,\beta n], \tau= \infty\}$ is bounded by  $Ce^{-\gamma n}$.
\end{lem}

Choose $\epsilon>0$ sufficiently small such that Lemmas \ref{ldlem}, \ref{taun} and \ref{En} are all satisfied. By simple set theory from Lemma \ref{taun} and Lemma \ref{En}, it is sufficient to prove that the probability of $\Big\{\sum\limits_{y \in Y} 1(y \in W_{n}^{0}) < \rho |Y|\Big\}$ on $\{[L_{n}, R_{n}] \supseteq [-\beta n,\beta n]\} \cap \{\tau = \infty\}$ is bounded by  $Ce^{-\gamma |Y|}$, this however follows from Lemma \ref{ldlem} by use of Lemma \ref{coupWW}.  
\end{proof}

We finally give a consequence of the last result. The argument is from the proof of Lemma 3 in \cite{MS}.

\begin{cor}
For all $\rho<1$ and $\beta<1$ there is $\epsilon>0$ such that for any $n\geq1$ and $b \in (0,\beta]$, 
the probability that there exists a sequence $(y_{k})_{k=1}^{b n}$ of consecutive points in $X(n) \cap [-\beta n,\beta n]$ such that $\sum\limits_{k=1}^{bn}1(y_{k} \in W_{n}^{0}) < \rho b n$ and $W_{n}^{0} \not= \emptyset$, is bounded by $C e^{-\gamma bn}$, where $C, \gamma >0$ are independent of $n$ and $b$. 
\end{cor}

\begin{proof} 
Since the number of $(y_{k})_{k=1}^{bn}$ considered is of polynomial order in both $n$ and $b$, the proof follows from Proposition \ref{1depmain}.
\end{proof}

\begin{remark}
\textup{The last corollary implies the corresponding statement for contact processes by use of the comparison result in \cite{D91}, and the argument in the proof of Proposition 3.3 in \cite{T} (equivalently, Proposition \ref{shadldev2} above). Alternatively this can be done by the arguments in the proof of Corollary 4 in \cite{MS}.}
\end{remark}

\chapter{On two basic monotonicity properties of three state contact processes}



\paragraph{Abstract:} The three state contact process is the modification of the contact process at rate $\mu$ in which first infections occur at rate $\lambda$ instead. It is shown that the condition $\mu \geq \lambda$ is necessary and sufficient for preserving monotonicity in set of initially infected sites analogously to the contact process. It is also shown that survival of the process for all $\mu$ and $\lambda$ is more likely than that of the process standard spatial epidemic at rate $\lambda$, that is the process for $\mu=0$. The proofs presented are based on elementary extensions of known coupling techniques. 
\vfill

\section{Introduction and results}
The \textit{three state contact process} on $G$, a connected graph of bounded degree, is a continuous time Markov process, $\zeta_{t}$, on the state space $\{-1,0,1\}^{V}$,  elements of which are called configurations. One can think of configurations as functions from $V$ to $\{-1,0,1\}$. The evolution of $\zeta_{t}$ is described locally as follows. Transitions at each site $u$, $\zeta_{t}(u)$, occur according to the rules:
\begin{equation*}\label{rates}
\begin{array}{cl}
-1 \rightarrow 1 & \mbox{ at rate } \lambda |\{v \sim u:\zeta_{t}(v)= 1\}| \\
\mbox{ } 0 \rightarrow 1 & \mbox{ at rate } \mu |\{v \sim u:\zeta_{t}(v)=1\} | \\   
\mbox{ } 1 \rightarrow 0 & \mbox{ at rate } 1,
\end{array}
\end{equation*}
for all $t\geq0$, where $v \sim u$ denotes that $u$ is joined to $v$ by an edge and, the cardinal of a set $B\subset V$ is denoted by $|B|$. We note that the cases that $\lambda = \mu$ and $\mu=0$ respectively correspond to the extensively studied contact process and to the forest fire model, see \cite{L99} and \cite{D88}. For general information about interacting particle systems, such as the fact that the assumption that $G$ being of bounded degree assures that the above rates define a unique process, see Liggett \cite{L}.

We incorporate the initial configuration, $\eta$, and the pair of the parameters $(\lambda,\mu)$ to our notation in the following fashion $\zeta_{t}^{\{\eta, (\lambda,\mu)\}}$. If $\eta$ is such that $\eta(x)=1$ for all $x\in A$, $A\subset V$, and  $\eta(x)=-1$ for all $x\in V \backslash A$ then we denote the process by $\zeta_{t}^{\{\eta, (\lambda,\mu)\}}$ as $\zeta_{t}^{\{A, (\lambda,\mu)\}}$, while if $A = \{u\}$ we abbreviate $\zeta_{t}^{\{\{u\}, (\lambda,\mu)\}}$ by 
$\zeta_{t}^{\{u, (\lambda,\mu)\}}$. 

The following result is the comparison with the forest fire model. 

\begin{prop}\label{ForestFire}
For any $(\lambda,\mu)$ and $w\in V$ we have that $\zeta_{t}^{\{w,(\lambda,\mu)\}}$ and $\zeta_{t}^{\{w,(\lambda,0)\}}$ 
can be coupled such that the event $\left\{\zeta_{t}^{\{w,(\lambda,0)\}}(v) = 1 \mbox{ for some } t\geq0  \right\}$ implies that $\left\{\zeta_{t}^{\{w,(\lambda,\mu)\}}(v) = 1 \mbox{ for some } t\geq0  \right\}$ for any $v\in V$, a.s..
\end{prop}

The proof of Proposition \ref{ForestFire} given below is based on locally dependent random graphs and is a variant of the arguments in Durrett \cite{D88}. An immediate consequence of Proposition \ref{ForestFire} is given next. For stating it, let us write $\{\zeta_{t} \mbox{ survives}\}$ to denote $\{\forall \mbox{ }t, \zeta_{t}(x) =1 \mbox{ for some } x \}$.  
\begin{cor}\label{compstdepid}
For any $(\lambda,\mu)$ and $w\in V$ we have that 
\begin{equation*}
\pr\left(\zeta_{t}^{\{w,(\lambda,0)\}} \textup{ survives}\right) \leq  \pr\left(\zeta_{t}^{\{w,(\lambda,\mu)\}} \textup{ survives}\right).
\end{equation*}
\end{cor}
We note that Corollary \ref{compstdepid} is a different version of Proposition 2 in Durrett and Schinazi \cite{DS} for general graphs. To establish those results the albeit elementary proofs given here are necessary because of the lack of (known) monotonicity properties when the parameters $(\lambda,\mu)$ are such that $\lambda > \mu$.

The next result is a generic monotonicity property. To state it let us endow the space of configurations $\{-1,0,1\}^{V}$ with the component-wise partial order $\mbox{i.e.}$, for any two configurations $\eta, \eta'$, such that $\eta \leq \eta'$ whenever $\eta(x) \leq \eta'(x)$ for all $x \in V$. 

\begin{thm}\label{PIPmon}
Let $\eta, \eta'$ be configurations and let $(\lambda, \mu)$, $(\lambda', \mu')$ be pairs of parameters. If $\eta \leq \eta'$ and $\lambda \leq \lambda'$, $\mu \leq \mu'$ and $\mu'\geq \lambda$ then $\zeta_{t}^{\{\eta, (\lambda, \mu)\}}$ and $\zeta_{t}^{\{\eta', (\lambda', \mu')\}}$ on $G$ can be coupled such that $\zeta_{t}^{\{\eta, (\lambda, \mu)\}} \leq \zeta_{t}^{\{\eta', (\lambda', \mu')\}}$ for all $t\geq0$ a.s..   
\end{thm}
The proof of Theorem \ref{PIPmon} given below is a variant of what is known as the basic coupling for interacting particle systems; see \cite{L99}. A proof of this property via a different approach is described in Stacey \cite{S}; see section 5 there. The following remark implies that the condition that $\mu'\geq \lambda$ in Theorem \ref{PIPmon} cannot be dropped. To the best of our knowledge the necessary proof of the following counterexample is not given elsewhere. 
\begin{remark}\label{remmon}
Let $G$ be the connected graph with $V = \{u,v\}$. Then: (i) for any $\lambda>1$, a coupling of $\zeta_{t}^{\{u, (\lambda,0)\}}$ and $\zeta_{t}^{\{V, (\lambda,0)\}}$ on $G$ 
such that $\zeta_{t}^{\{u, (\lambda,0)\}} \leq \zeta_{t}^{\{V, (\lambda,0)\}} \mbox{for all } t \geq0$ cannot be constructed; additionally, (ii) for all $\lambda, \lambda'$, if $\lambda < \lambda'<1$ then  a coupling of $\zeta_{t}^{\{u, (\lambda,0)\}}$ and $\zeta_{t}^{\{u, (\lambda',0)\}}$ on $G$ such that $\zeta_{t}^{\{u, (\lambda,0)\}} \leq \zeta_{t}^{\{u, (\lambda',0)\}} \mbox{for all } t \geq0$ cannot be constructed.
\end{remark}

\section{Proofs}

\begin{proof}[proof of Proposition \ref{ForestFire}]
We give a specific construction of $\zeta_{t}^{\{w,(\lambda,\mu)\}}$ on $G$. For this it is useful to introduce the following epidemiological interpretation, we regard site $x$ as infected if $\zeta_{t}^{\{w,(\lambda,\mu)\}}(x) = 1$, as susceptible and never infected if $\zeta_{t}^{\{w,(\lambda,\mu)\}}(x) = -1$ and, as susceptible and previously infected if $\zeta_{t}^{\{w,(\lambda,\mu)\}}(x) = 0$. Thus, for $\zeta_{t}^{\{w,(\lambda,\mu)\}}$ transitions $-1 \rightarrow 1$,  $ 0 \rightarrow 1$, and $1 \rightarrow 0$ are thought of as initial infections, secondary infections and recoveries respectively. 

For all $u\in V$ let $(T_{n}^{u})_{n\geq1}$ be exponential 1 r.v.'s; further for all $(u,v)$ such that $u\sim v$, let $(Y_{n}^{(u,v)})_{n\geq1}$ be exponential $\lambda$  r.v.'s and $(N_{n}^{(u,v)})_{n\geq1}$ be Poisson processes at rate $\mu$. All random elements introduced are independent while $\pr$ denotes the probability measure on the space on which these are defined. To describe the construction let $\tau_{k,n}^{(u,v)}, k\geq1$, be the times of events of $N_{n}^{(u,v)}$ within the time interval $[0,T_{n}^{u})$ and let also $X_{n}^{(u,v)}, n\geq1$, be such that $X_{n}^{(u,v)} = Y_{n}^{(u,v)} $ if $Y_{n}^{(u,v)} < T_{n}^{u}$ and $X_{n}^{(u,v)} := \infty$ otherwise. 

We construct $\zeta_{t}^{\{w,(\lambda,\mu)\}}$ on $G$ as follows. Suppose that site $u$ gets infected at time $t$ for the $n$-th time, $n\geq1$ then: (i) at time $ t + T_{n}^{u}$ a recovery occurs at site $u$, (ii) at time $t+X_{n}^{(u,v)}$ an initial infection of $v$ occurs if immediately prior to that time site $v$ is at a never infected state, and, (iii) at each time $t+\tau_{k,n}^{(u,v)}, k\geq1$, a secondary infection occurs at site $v$ if immediately prior to that time site $v$ is at a susceptible and previously infected state.  

Let also $\mathcal{X}_{u} = \{v: v\sim u \mbox{ and } X_{1}^{(u,v)}<\infty\}$, for all $u\in V$. Let $\overrightarrow G$ denote the directed graph produced from $G$ by replacing each edge between two sites $u,v$, $u\sim v$, with two directed ones, one from $u$ to $v$ and one from $v$ to $u$. Let also  $\Gamma$ denote the subgraph of $\overrightarrow {G}$ produced by retaining edges from $u$ to $v$ only if $v \in \mathcal{X}_{u}$, for all $u,v \in  V$, and let $u \mbox{ }_{\overrightarrow{(\mathcal{X}_{u}, u\in  V)}}  \mbox{ }v $ denote the existence of a directed path from $u$ to $v$ in $\Gamma$. By the construction given above for $\zeta_{t}^{\{w,(\lambda,0)\}}$ and the proof of Lemma 1 in Durrett \cite{D88}, Chpt.\ 9, we have that 
\begin{equation*}\label{Czeta2}
\{w \mbox{ }_{\overrightarrow{(\mathcal{X}_{u}, u\in  V)}}  \mbox{ } v \}  = \left\{\zeta_{t}^{\{w,(\lambda,0)\}}(v) = 1  \mbox{ for some } t\geq0  \right\}, 
\end{equation*}
and similarly for $\zeta_{t}^{\{w,(\lambda,\mu)\}}$ we also have that 
\begin{equation*}\label{Czeta}
\{w \mbox{ }_{\overrightarrow{(\mathcal{X}_{u}, u\in  V)}}  \mbox{ } v \}  \subseteq \left\{\zeta_{t}^{\{w,(\lambda,\mu)\}}(v) = 1  \mbox{ for some } t\geq0  \right\}, 
\end{equation*}
for all $v\in  V$. The proof is complete by combining the two final displays. 
\end{proof}

\begin{proof}[proof of Corollary \ref{compstdepid}] Consider the process $\zeta_{t}^{\{w,(\lambda,\mu)\}}$, letting $A_{v}$ denote the event $\{\zeta_{t}^{\{w,(\lambda,\mu)\}}(v) =1\mbox{ for some } t\geq0\}$, $v \in  V$, from Proposition \ref{ForestFire} we have that the proof is completed by the following equality,
\begin{equation*}\label{onsurvinfclust}
\pr\left(\sum_{v \in  V} 1(A_{v}) = \infty \right) = \pr\big(\zeta_{t}^{\{w,(\lambda,\mu)\}}\mbox{ survives}\big),
\end{equation*}
where $1(\cdot)$ denotes the indicator function. To prove the display above let $B_{M}$ denote the event  $\big\{ \sum_{v \in  V} 1(A_{v}) \leq M \big\}$, for all $M\geq1$, and note that, by elementary properties of exponential random variables, we have that $\pr\big(B_{M}, \zeta_{t}^{\{w,(\lambda,\mu)\}}\mbox{ survives}\big) =0$, and thus $\pr\Big(\bigcup_{M\geq1} B_{M},  \zeta_{t}^{\{w,(\lambda,\mu)\}}\mbox{ survives}\Big) = 0$.
\end{proof}






\begin{proof}[proof of Theorem \ref{PIPmon}]
Let us simplify notation, we write $\zeta_{t}$ for $\zeta_{t}^{\{\eta, (\lambda, \mu)\}}$ and  $\zeta_{t}'$ for $\zeta_{t}^{\{\eta', (\lambda', \mu')\}}$. We use the coupling that for all $x\in V$ has the following transitions for $(\zeta_{t}'(x), \zeta_{t}(x))$, 
  
\begin{equation*}
(0, -1) \rightarrow \hspace{2mm}
\begin{cases}
(1,1) & \text{at rate } \lambda|y\sim x: \zeta_{t}(y) = 1| \\
(1,-1) & \text{at rate } \mu' |y\sim x: \zeta'_{t}(y) = 1| - \lambda |y\sim x: \zeta_{t}(y) = 1|
\end{cases}
\end{equation*}
\begin{equation*}
(-1, -1) \rightarrow 
\begin{cases}
(1,1) & \text{at rate } \lambda|y\sim x: \zeta_{t}(y) = 1| \\
(1,-1) & \text{at rate } \lambda' |y\sim x: \zeta'_{t}(y) = 1| - \lambda |y\sim x: \zeta_{t}(y) = 1|
\end{cases}
\end{equation*}

\begin{equation*}
(0, 0) \rightarrow \hspace{2mm}
\begin{cases}
(1,1) & \text{at rate } \mu|y\sim x: \zeta_{t}(y) = 1| \\
(1,0) & \text{at rate } \mu' |y\sim x: \zeta'_{t}(y) = 1| - \mu |y\sim x: \zeta_{t}(y) = 1|
\end{cases}
\end{equation*}
Further, $(1, -1) \rightarrow (1,1)$ at rate $\lambda |y\sim x: \zeta_{t}(y) = 1 |$ while 
$(1, -1) \rightarrow (0,-1)$ at rate 1. Also, $ (1, 0) \rightarrow (1,1)$ at rate $\mu |y\sim x: \zeta_{t}(y) = 1 |$ while $ (1, 0) \rightarrow (0,0)$ at rate 1. Finally,  $(1, 1) \rightarrow (0,0)$ at rate 1.
\end{proof}





\begin{proof}[proof of Remark \ref{remmon}]
Let $T_{u},T_{v}$ be exponential 1 r.v.'s; let also $X_{u,v}$ be an exponential $\lambda$ \mbox{r.v.} and $f_{X_{u,v}}$ be its probability density function. All r.v.'s introduced are independent of each other and defined on some probability space with probability measure $\pr$. We have that for any $t\geq0$   
\begin{eqnarray}\label{calc}
\pr\left(\zeta_{t}^{\{u, (\lambda,0)\}} = (1,1) \right) & = &  \pr(T_{u}>t) \int_{0}^{t} f_{X_{u,v}}(s)\pr(T_{v}> t-s)\, ds \nonumber\\
& = & e^{-2t}  \int_{0}^{t} \lambda e^{s (1- \lambda)} \, ds \nonumber \\
& =& e^{-2t}  \frac{\lambda}{\lambda-1} (1-e^{-t(\lambda-1)}).
\end{eqnarray}

By $(\ref{calc})$, note that (a) for all $\lambda >1$ we can choose $t$ sufficiently large, i.e. $\displaystyle{ t > \frac{\log{\lambda}}{\lambda-1}}$, such that $\pr\left(\zeta_{t}^{\{u, (\lambda,0)\}}=(1,1)\right) > e^{-2t} =  \pr\left(\zeta_{t}^{\{V, (\lambda,0)\}} = (1,1) \right)$; and note further that (b) for all $\lambda <1$, $\pr\left(\zeta_{t}^{\{u, (\lambda,0)\}} = (1,1)\right)$ is not an increasing function of $\lambda$. By Theorem B9 in Liggett \cite{L99}, (a) and (b) imply respectively the first and second parts of the remark statement. 
\end{proof}


\end{document}